\documentclass{article}

\usepackage{graphicx}
\usepackage{amsmath,amsthm,amssymb}
\usepackage{color}
\usepackage{url}

\newtheorem{theorem}{Theorem}[section] %% blank out with SV class
\newcommand{\newsec}[1]{\setcounter{equation}{0}\section{#1}}

\newcommand{\be}{\begin{equation}}
\newcommand{\ee}{\end{equation}}
\newcommand{\eq}[1]{(\ref{eq:#1})}

\newcommand{\bi}{\begin{itemize}}
\newcommand{\ei}{\end{itemize}}

\newcommand{\ben}{\begin{enumerate}}
\newcommand{\een}{\end{enumerate}}

\newcommand{\noi}{\noindent}

\newcommand{\C}{\mathbb{C}}
\newcommand{\N}{\mathbb{N}}
\newcommand{\R}{\mathbb{R}}
\newcommand{\Z}{\mathbb{Z}}

\newcommand{\floor}[1]{\lfloor #1 \rfloor}

\newcommand{\ul}[1]{\underline{#1}}
\renewcommand{\vec}[1]{\mbox{\boldmath $#1$}}
\def\tens#1{\ensuremath{\mathsf{#1}}}

\newcommand{\ds}[1]{{\displaystyle #1}}

\newcommand{\eps}{\varepsilon}

\newcommand{\wht}[1]{\widehat{#1}}

\newcommand{\st}{\rule[-1.5ex]{0mm}{4ex}}

\newcommand{\cutout}[1]{}

\newtheorem{defn}{Definition}[section]

\newcommand{\dt}{h}
\newcommand{\dx}{\Delta x}
\newcommand{\oo}{{\cal O}}

\newcommand{\tp}{t'}
\newcommand{\pole}{s}
\newcommand{\intK}[1]{K^{(-#1)}}

\newcommand{\lap}[1]{\overline{#1}}
\newcommand{\w}{\omega}
\newcommand{\q}{q}
\newcommand{\vv}{v}
\newcommand{\U}{U}
\newcommand{\V}{V}
\newcommand{\Q}{Q}

\newcommand{\Ztrans}[1]{{\cal Z}\{#1\}}

\newcommand{\x}{x}
\newcommand{\y}{y}
\newcommand{\symb}{B}

\newcommand{\db}{\phi^*}

\newcommand{\vecx}{\vec{x}}
\newcommand{\vecy}{\vec{y}}
\newcommand{\vecw}{\vec{\w}}

\newcommand{\cumulmax}{Z}
\newcommand{\Md}{M}
\newcommand{\ele}{\Gamma}

\title{Convolution spline approximations for time domain boundary integral equations}

\author{Penny J
Davies\thanks{Department of Mathematics and Statistics, University
of Strathclyde, 26 Richmond St, Glasgow, \mbox{G1 1XH}, UK;
\texttt{penny.davies@strath.ac.uk}} \and Dugald B
Duncan\thanks{Maxwell Institute for Mathematical Sciences,
Department of Mathematics, Heriot--Watt University,
Edinburgh, \mbox{EH14 4AS}, UK;
\texttt{D.B.Duncan@hw.ac.uk} (address for correspondence)}}

\date{\textbf{To appear in Journal of Integral Equations and Applications\\ \ \\ This version: \today}\\ \ \\ (First version: April 6, 2012)}

\begin{document}

\maketitle

\begin{abstract}
\noi
We introduce a new ``convolution spline'' temporal approximation of
time domain boundary integral equations (TDBIEs).  It shares some properties of convolution quadrature (CQ), but instead of being based on an underlying ODE solver the approximation is explicitly
constructed in terms of compactly supported basis functions.  This results in sparse system matrices and makes it computationally more efficient than using the linear multistep version of CQ for TDBIE time-stepping.  We use a Volterra integral equation (VIE) to illustrate the derivation of this new approach: at time step $t_n = n\dt$ the VIE solution is approximated in a backwards-in-time manner in terms of basis functions $\phi_j$
by $u(t_n-t) \approx \sum_{j=0}^n u_{n-j}\,\phi_j(t/\dt)$ for $t \in [0,t_n]$.
We show that using isogeometric B-splines of degree $m\ge 1$ on $[0,\infty)$ in this framework gives a second order accurate scheme, but cubic splines with the parabolic runout conditions at $t=0$ are fourth order accurate.  We establish a methodology for the stability analysis of VIEs and demonstrate that the new methods are stable
for non-smooth kernels which are related to convergence analysis for TDBIEs, including the case of a Bessel function kernel oscillating at frequency $\oo(1/\dt)$.  Numerical results for VIEs and for TDBIE problems on both open and closed surfaces confirm the theoretical predictions.
\smallskip

\noi \textbf{Keywords:} Convolution quadrature, Volterra integral equations, time dependent boundary integral equations
\smallskip

\noi
\textbf{AMS(MOS) subject classification:} 65R20, 65M12

\end{abstract}

%=====================================================================
%=====================================================================
\newsec{Introduction}
\label{sec:intro}

Convolution quadrature (CQ) time-stepping for time-dependent boundary integral equations (TDBIEs) was first proposed and analysed by Lubich in 1994 \cite{Lu1994}.  Since then the inherent stability and ease of implementation of CQ (as compared to a full space-time Galerkin approximation) has made it a very popular choice for TDBIE problems -- a search on \texttt{"convolution quadrature" "boundary"} in the Thomson Reuters Web of Science database yields nearly 200 hits.  Unfortunately there is a drawback: the effective support of the time basis functions $\phi_j(t)$ which underpin CQ increases with $j$, and this increases the computational complexity of the solution algorithm.
Here we describe a new ``convolution spline'' approximation framework which shares some properties with CQ, but is explicitly constructed in terms of compactly supported basis functions which are (mainly) translates -- this makes it easy to implement and computationally efficient.
We apply it to the TDBIE problem
 \be
  \frac{1}{4 \pi}\, \int_\Gamma \frac{u(\vecx',t\!-\!|\vecx'\!-\!\vecx|)}
  {|\vecx'\!-\!\vecx|}\; d\vecx' = a(\vecx,t)\quad
  \mbox{for $\vecx \in \Gamma$, $t > 0$}
  \label{eq:tdbie}
\ee
for $u$ -- this is the single layer potential equation for acoustic scattering
from the surface $\Gamma \subset \R^3$ with zero Dirichlet boundary
conditions and (known) incident field $-a(\vecx,t)$, which is equivalent to
\[
  \int_0^t \int_\Gamma k(\vecx'\!-\!\vecx, t-t')\, u(\vecx',t')\; d\vecx'\, dt' = a(\vecx,t)\quad \mbox{for}\quad k(\vec{z},t) = \frac{\delta(t-|\vec{z}|)}{4 \pi\, |\vec{z}|}\;.
\]

We use the convolution--kernel Volterra integral equation (VIE)
\be
    \int_0^t K(\tp)\, u(t-\tp)\, d\tp = a(t)\,,
  \quad t \in [0,T]\,
\label{eq:vie}
\ee
to illustrate the derivation of the new approximation method and its
convergence and stability properties.  However, the focus of the paper is not on deriving new methods for VIEs (of which there are already very many), but on using the insight gained from VIEs to derive new methods which have good properties for TDBIEs.

%=====================================================================
\subsection{Properties of TDBIE approximations}
\label{sec:intro.tdbie}

Designing a good approximation scheme for the TDBIE \eq{tdbie} is nontrivial;
challenges include ensuring that it is numerically stable, it is not prohibitively hard to implement for a given scattering surface $\Gamma$, and its computational complexity is not infeasibly high.  We begin by briefly summarising the pros and cons of some of the main approaches (see also \cite{Co2004,HaD2003}).

Bamberger and Ha~Duong \cite{BaHaD1986} proved that a full Galerkin approximation of \eq{tdbie} in time and space is stable and convergent for smooth, closed $\Gamma$ (this was extended to the case of open, flat $\Gamma$ in \cite{HaD1990}), but the stability of the
method relies on all the integrals being evaluated very accurately
(the key insight on how to do this was provided by
Terrasse \cite{TerrassePhD}).  In practice this involves converting five dimensional volume integrals over irregular (non-polygonal)
sub-regions of $\Gamma \times \Gamma \times [0,T]$ to surface integrals which are then evaluated using high precision quadrature,
and is extremely complicated to successfully implement in practice, even for relatively simple $\Gamma$.
Collocation schemes for \eq{tdbie} are far more straightforward to implement, but there is little rigorous convergence analysis for them, and numerical instability is often an issue.  As noted above, methods which use a Galerkin approximation in space and CQ in time have obvious attractions: they are based on rigorous theoretical analysis \cite{BaHaD1986,Lu1994} (see also \cite{Do-Sa2011} for some new bounds) and are relatively straightforward to implement.  They are also inherently far more stable than those which use Galerkin or collocation time approximations (Lubich showed in \cite{Lu1994} that the CQ method remains stable when the inner product integrals are approximated), but unfortunately the disadvantage this time is higher computational complexity.

All three approaches approximate \eq{tdbie} as a convolution sum of the form
$\sum_{j=0}^n \tens{Q}^j\, \ul{U}^{n-j} = \ul{a}^n$, which is rearranged
to give the time-stepping scheme
 \be
  \tens{Q}^0\, \ul{U}^n = \ul{a}^n - \sum_{j=1}^n \tens{Q}^j\, \ul{U}^{n-j}
\label{eq:BIEscheme}
 \ee
for $\ul{U}^n \in \R^{N_S}$, the representation of the spatial
approximation of $u$ at or near time $t^n = n\, \dt$, where the
right-hand side vector $\ul{a}^n$ is derived from $a(\vecx,t)$.   In
the case of both Galerkin and collocation approximations the
matrices $\tens{Q}^j \in \R^{N_S \times N_S}$ are sparse -- the number of
nonzero elements per row of matrix $\tens{Q}^j$ is $\oo(\min\{j,\,
N_S^{1/2}\})$.  In particular this means that \eq{BIEscheme} can be
solved in $\oo(N_S^{3/2})$ operations once the right-hand side is known, and
the overall computational complexity to obtain the approximate
solution up to time $N_T\, \dt$ is $\oo(\min\{N_T^3\, N_S,\, N_T^2\,
N_S^{3/2}\})$ operations.  For these hyperbolic problems it is usual
to use a timestep $\dt$ commensurate with the side $\dx$ of a
typical space mesh element, and in this case $N_T \approx N$ and
$N_S \approx N^2$ for $N = 1/\dx$, and the total computational
complexity is $\oo(N^5)$.  Although this compares somewhat
unfavourably with the $\oo(N^4)$ computational complexity of a
finite difference or finite element approximation of the PDE
formulation of the acoustic wave equation in $\R^3$, the plane wave
``fast'' methods developed by Michielssen and co-workers
\cite{ErShMi1999,ErShMi2000,LuWaErMi2000}
reduces the complexity to $\oo(N^3\, \log^2 N)$.

Using CQ in time results in a solution algorithm \eq{BIEscheme} in which the matrices
$\tens{Q}^j$ are
dense, because the underlying basis functions are global (see e.g.\ Sec.~\ref{sec:CQ} below, or \cite{Ba2010,HaKrSa2009} for more
details), which increases the computational complexity to $\oo(N_S^2\, N_T^2)$. The issue is
not solving \eq{BIEscheme} for $\ul{U}^n$ (which can typically be
done efficiently by approximating $\tens{Q}^0$ appropriately), but in
performing the matrix--vector products needed to calculate the
right-hand side. Lubich explains that the technique of
\cite{HaLuSch} can be used to reduce the overall complexity to
$\oo(N_S^2\, N_T\, \log^2 N_T)$, i.e.\ $\oo(N^5\, \log^2 N)$.
A cut-off strategy to replace small matrix
entries by zero is described and analysed in
\cite{HaKrSa2009}, and this reduces the storage costs of the method.
This is combined with panel clustering in \cite{KrSa2008} to further
reduce the storage costs.  However, because the effective support of the time basis functions increases with index (see Fig.~\ref{fig:basis_BBDF2} or \cite[Fig.~2.2]{cq2}), the computational complexity is a factor of $\sqrt{N}$ higher than that for approximations which use local basis functions.

CQ methods which are based on underlying Runge--Kutta ODE solvers
have also been developed and analysed for TDBIEs
\cite{BaLu2011,BaLuMe2011}.  There are several advantage of these
methods over linear multistep CQ methods: the basis functions are
more highly concentrated \cite[Figs 1--2]{Ba2010}, which makes
sparsifying the $\tens{Q}^j$ matrices more straightforward; and higher
order accurate methods in time are possible.  Banjai \cite{Ba2010} uses this
approach to develop a practical, parallelizable
solution algorithm for \eq{tdbie} which he illustrates with a number
of realistic large-scale numerical examples.

%=====================================================================
\subsection{New convolution spline methods}
\label{sec:intro.cspl}

The $\tens{Q}^j$ system matrices in \eq{BIEscheme} for our new method have the same sparsity pattern as for the Galerkin or collocation approximations described above, and so it is considerably more efficient (both to set up by calculating the system matrices, and to run) than using the linear multistep version of CQ.  Our method gives
a TDBIE solution scheme whose overall complexity is
$\oo(\min\{N_T^3\, N_S,\, N_T^2\, N_S^{3/2}\}) = \oo(N^5)$ operations (and which could also be potentially speeded up using fast methods).  It is also far easier to implement than the full space--time Galerkin approach.

We derive the new approximation as a solution method for the VIE \eq{vie}, with $u$ approximated in terms of B-spline basis functions in a backwards-in-time framework.  Our initial approach is to use isogeometric B-splines of degree $m$ on $[0,\infty)$.  There can be advantages in using higher order values of $m$ even though the formal convergence rate of this scheme for a smooth VIE problem is limited to second order (because it is based on quasi-interpolation by the Schoenberg B-spline operator).  For example,
as noted in
\cite{SaVe2011b}, using smooth temporal basis functions greatly simplifies
approximating the integrals in \eq{tdbie}.
We also consider cubic B-splines with the parabolic runout condition at $t=0$ and show that these are fourth order accurate.
We carefully test out the new methods on \eq{vie}, establishing formal convergence, and examining the behaviour for kernels which mimic some of the important properties of TDBIE problems, such as discontinuous step-function kernels (see e.g.\ \cite{SaVe2011a}).  Another important test problem is obtained from taking the spatial Fourier transform of \eq{tdbie} at frequency
$\vecw \in \R^2$ when $\Gamma =
\R^2$.  This is
\be
  \int_0^t J_0(\w t')\, \wht{u}(\vecw, t-t')\, dt' =  2
  \wht{a}(\vecw, t)\,,
\label{eq:BFvie}
\ee
where $\w = |\vecw|$ and $J_0$ is the first kind Bessel function of order zero.  As noted in \cite{rpie}, instabilities of approximation schemes for \eq{tdbie} are typically exhibited at the highest spatial frequency which can be represented on the mesh.  Hence it is important to ensure that any prototype numerical scheme for time-stepping \eq{tdbie} is stable for \eq{BFvie} at values of $\w = \oo(1/\dt)$ (assuming $\dt \approx \dx$).

%=====================================================================
\subsection{Outline}
\label{sec:intro.plan}

Section \ref{sec:CQ} contains an alternative derivation of Lubich's \cite{Lu1988a} CQ method for \eq{vie} in terms of basis functions which have the sum to unity property \eq{pou}.  The new convolution spline approximation of \eq{vie} is described in Section \ref{sec:spline} in terms of basis functions which have compact support and are (essentially) all translates, and we give sufficient conditions for this approximation to be stable.  We consider the case in which the basis functions are $m$th degree isogeometric B-splines on $[0, \infty)$ in Section \ref{sec:bspline}, showing how Laplace transform techniques can be used to prove the stability of this approximation of \eq{vie} for several different test kernels, and demonstrating second order convergence for \eq{vie} when $K$ and $a$ satisfy
\be \label{eq:smoothness}
  a \in C^{d+1}[0,T]\,, \ K \in C^{d+1}[0,T]\,, \ a(0)=0 \
  \mathrm{and} \ K(0) = 1
\ee
for suitable $d \ge 0$.  Under these assumptions, equation \eq{vie}
possesses a unique solution $u \in C^{d}[0,T]$ --  e.g.\ see,
\cite[Theorem 2.1.9]{brunnerbook}.

In Section~\ref{sec:modbs} we consider a cubic convolution spline basis which is modified near $t = 0$ to satisfy the parabolic runout conditions, and show that this gives a far more stable approximation of \eq{vie} which is fourth order convergent.  Numerical tests  show that it achieves fourth order accuracy even for a discontinuous kernel.
We present numerical test results for TDBIEs in Section~\ref{sec:tdbie} which use a Galerkin approximation in space (based on triangular piecewise constant elements), and the new cubic convolution spline basis in time, for both open and closed surfaces $\Gamma$.  These show that the new scheme performs far better than CQ based on BDF2 -- it is both more accurate and more efficient.

The TDBIE test problems are similar to those considered in \cite{cq2} which
use the convolution--in--time framework with non-polynomial (global) basis functions, but the modified B-spline basis functions give a more accurate temporal approximation.  We note that the time-stepping schemes of \cite{cq2} rely on the theoretical framework developed in Sections~2--3 of the present work.

%=====================================================================
%=====================================================================
\newsec{CQ based on linear multistep methods for \eq{vie}}
\label{sec:CQ}

We begin by outlining Lubich's derivation \cite{Lu1988a} of the
CQ method for \eq{vie} in order to show how it
can be reinterpreted in terms of CQ basis functions.  For simplicity
we restrict attention to the case for which the extension of the
solution $u$ by zero to the negative real axis is in
$C^{d}(-\infty,T]$ (otherwise the CQ method needs to be `corrected'
as described in \cite[Sec.\ 3]{Lu1988a} in order to attain optimal
convergence). This is guaranteed by requiring
 \be
  a^{(p)}(0) = 0 \quad \mbox{for $p = 0:d+1$}\,
\label{eq:zerobc}
 \ee
because $u^{(p)}(0) = a^{(p+1)}(0) - \sum_{\ell=0}^{p-1} K^{(p-\ell)}(0)\; u^{(\ell)}(0)$.
We also assume that the Laplace transform $\lap{K}(s)$ of the kernel $K$ is sufficiently well-behaved for all the formal manipulations in the next subsection to be rigorous.  For details see for example \cite[App]{Ba2010} or \cite[Sec.\ 1]{Lu1994}.

%=====================================================================
\subsection{Lubich's CQ method}
\label{sec:CQ.lub88}

We follow Lubich \cite{Lu1988a} and substitute the Laplace inversion formula for $\lap{K}(s)$ into \eq{vie}
to obtain
\be
  a(t) = \frac{1}{2\pi i} \int_\gamma \lap{K}(s)\, y(t,s)\, ds\,,
\label{eq:vielap}
\ee
where $\gamma$ is an infinite contour within the region of analyticity of $\lap{K}(s)$ and
$\ds{y(t,s) = \int_0^{t} e^{s\tp}\, u(t\!-\!\tp)\, d\tp\,.}$
Treating the Laplace variable $s$ as a parameter, $y(t)$ solves the ODE:
\be
 \dot{y}(t) = s\, y(t) + u(t)\,, \quad y(0) = 0\,,
\label{eq:ode}
 \ee
and this is approximated by the $k-$step ($k \le d$) linear
multistep method with timestep $\dt$
 \be
  \sum_{j=0}^k \alpha_j\, y_{n+j-k} = \dt \sum_{j=0}^k \beta_j\, f_{n+j-k}\,,
\label{eq:lmm}
 \ee
where $t_n = n \dt$, $y_n \approx y(t_n)$ and $f_n = s\, y_n +
u(t_n)$.   The starting values are $y_{-k} = \dots y_{-1} = 0$
because of the assumption \eq{zerobc}.  Multiplying \eq{lmm} by
$\xi^n$ and summing over $n$ (for $\xi \in \C$ for which the sum converges) gives
\[
  \left( \frac{\delta(\xi)}{h} - s \right)\, \sum_{n=0}^\infty y_n\, \xi^n = \sum_{n=0}^\infty u(t_n)\, \xi^n\,, \quad\mbox{where}\quad \delta(\xi) = \left.\sum_{j=0}^k \alpha_j\, \xi^{k-j} \right/ \sum_{j=0}^k \beta_j\, \xi^{k-j}
\]
is the symbol of \eq{lmm}.
Hence $y_n$ is the coefficient of $\xi^n$ in the expansion of
$\left( \frac{\delta(\xi)}{h} - s \right)^{-1} \sum_{k=0}^\infty u(t_k)\, \xi^k$\,.
Substituting $y_n$ for $y(t_n)$ in \eq{vielap} shows that $a(t_n)$ is approximated by the coefficient of $\xi^n$ in
\[
  \frac{1}{2\, \pi\, i}\, \int_\gamma \left( \frac{\delta(\xi)}{h} - s \right)^{-1} \lap{K}(s)\, ds\, \sum_{k=0}^\infty u(t_k)\, \xi^k = \lap{K}(\delta\left(\xi)/\dt\right)\, \sum_{k=0}^\infty u(t_k)\, \xi^k\,
\]
using Cauchy's integral formula.  Hence, defining the CQ weights $\q_k = \q_k(h)$ to be the coefficients in the expansion
\be
  \lap{K}(\delta\left(\xi)/\dt\right) = \sum_{k=0}^\infty \q_k\, \xi^k
\label{eq:defw}
\ee
gives the CQ approximation of \eq{vie}
\be
  a(t_n) = \sum_{j=0}^n \q_j\, u_{n-j}\,.
\label{eq:scheme}
 \ee
This can be
rearranged to give the time-stepping approximate solution $u_n \approx u(t_n)$
\be \label{eq:timemarching}
  u_n = \frac{1}{\q_0}\, \left(a(t_n) - \sum_{j=1}^{n-1} \q_j\, u_{n-j}\right) \quad \mbox{for $n \ge 1$\,,}
\ee
since by assumption $u_0 = u(0) = 0$.

%=====================================================================
\subsection{Derivation of CQ in terms of basis functions}
\label{sec:CQ.ex}

The CQ approximation scheme \eq{scheme} for the VIE \eq{vie} is defined solely in terms of the weights $\q_k$.  But if CQ is used to time-step a TDBIE, then the approximation involves CQ basis functions -- see e.g.\  \cite{Ba2010,HaKrSa2009,MoScSt2011}.  However, we are not aware of a general interpretation of CQ approximation schemes for \eq{vie} in terms of basis functions. As well as yielding some interesting observations, this also gives the framework which we use for the derivation of our convolution spline methods in Sections \ref{sec:bspline}--\ref{sec:modbs}

At $t = t_n := n\dt$ \eq{vie} can be written as
\be
   a(t_n) = \int_0^\infty K(\tp)\, u(t_n-\tp)\, d\tp\,,
\label{eq:vie-inf}
 \ee
because $u(t) = 0$ for $t \le 0$.
We show below that the standard CQ method is
equivalent to approximating $u$ in \eq{vie-inf} by
 \be
  u(t_n - t') \approx \sum_{j=0}^n\, u_{n-j}\, \phi_j(t'/\dt)\quad
  \mbox{for $t' \ge 0$}
\label{eq:CQapprox}
 \ee
where $\phi_j$ are basis functions, i.e.\ the approximation at $t_n$ is $U_n(t) = \sum_{k=0}^n u_k\, \phi_{n-k}(n-t/\dt)$ for $t \le t_n$.  Note that $\phi_{n-k}(n-t/\dt)$ depends on $n$ -- i.e.\ CQ
is 
\textbf{fundamentally different} from a standard
finite--element type approximation in which an unknown coefficient is always associated with the same basis function.

Substituting \eq{CQapprox} into \eq{vie-inf} and comparing the
resulting expression with \eq{scheme} gives the relationship between the standard CQ weights
and basis functions:
 \be
  \q_j = \int_0^\infty K(t)\, \phi_j(t/\dt)\, dt\,.
\label{eq:wtdef}
 \ee
Comparing this with the
standard CQ definition of $\q_j$ in \eq{defw} gives (see \cite[Eq.\ (3.1)]{Ba2010})
 \be
  e^{-\delta(\xi) t} = \sum_{j=0}^\infty \phi_j(t)\, \xi^j\,.
\label{eq:defphi}
 \ee

An immediate consequence is that the basis functions satisfy the sum to unity property
\be
  \sum_{j=0}^\infty \phi_j(t) = 1\,,
\label{eq:pou}
 \ee
provided the underlying multistep ODE solver is consistent, because in this case $\delta(1)
= 0$.  This new observation is a crucial property which we use in Section \ref{sec:spline}.

%=====================================================================
\subsection{CQ basis functions for LMMs}
\label{sec:CQ.eg}

Explicit formulae for the $\phi_j(t)$ based on BDF1--2 have been used for TDBIE approximations \cite{HaKrSa2009,MoScSt2011}.  The formula for BDF1 is given in \cite{MoScSt2011}, and in this case
$\phi_j(t) = e^{-t}\, t^j/j!$,
i.e.\ they are Erlang functions, used in statistics as probability density functions and satisfy $\phi_j(t) \ge 0$ and $\int_0^\infty \phi_j(t)\, dt = 1$.  The derivation for
BDF2 is more complicated, and the explicit formula
\[
         \phi_j(t) = \frac{1}{j!}\, H_j(\sqrt{2t}) \left(\frac{t}{2}\right)^{j/2}e^{-3t/2}
\]
is given in \cite{HaKrSa2009}, where $H_j$ is the $j$th Hermite polynomial.
Note that the properties of $H_j$ imply that
$\phi_j(t)$ involves a $j$th degree polynomial and an exponential in $t$
with no fractional powers of $t$.

In principle \eq{defphi} can be used directly to find the basis functions
$\phi_j(t)$ corresponding to any underlying linear multistep ODE
method for \eq{ode}, although this may not be easy in practice.  For the trapezoidal rule $\delta(\xi) = 2\,(1-\xi)/(1+\xi)$ \cite{Ba2010} and \eq{defphi} is $\sum_{j=0}^\infty \phi_j(t)\, \xi^j = e^{-2\, t}\, f(\xi)$,
where $f(\xi) = \exp(4\, t/(1+\xi))$.  This gives $f(\xi) = \sum_{j=0}^\infty f_j\, \xi^j$ where
\[
  f_j = \left.\frac{1}{j!}\, \frac{d^j}{d\xi^j}\, f(\xi) \right|_{\xi =
  0} = \left.\frac{1}{j!\, (4\, t)^j}\, \frac{d^j}{dz^j}\, e^{-1/z} \right|_{z =
  1/(4t)}
\]
using the change of variables $z = (1 + \xi)/(4t)$.  It follows from
\cite[eq.\ 18.5.6]{nist} that $f_j = (-1)^j\, e^{-4\, t}\, L_j^{-1}(4\,t)$\,,
where $L_j^\alpha(x)$ is a Laguerre polynomial.
The identity $L_j^{-1}(x) = L_j(x) - L_{j-1}(x)$ \cite[eq.\ 8.971--5]{grad-ryzh}
gives the trapezoidal rule basis functions $\phi_j(t) = (-1)^j\, \left\{ \ell_j(4\,t) - \ell_{j-1}(4\,t)\right\}$,
where $\ell_j(x) = e^{-x/2}\, L_j(x)$ is the $j$th Laguerre
function.
They are oscillatory, but do satisfy $\int_0^\infty \phi_j(t)\, dt = 1$.
The low order
basis functions are shown in Fig.~\ref{fig:basiscompare} (see also \cite[Fig.\ 1]{Ba2010} and \cite[Fig.~4]{MoScSt2011}).  Fig.~\ref{fig:basis_BBDF2} shows how the CQ basis functions spread out as $j$ increases -- this increases the number of non-zero entries in the $\tens{Q}^j$ matrices of \eq{BIEscheme} and makes CQ time-stepping less efficient.

The direct approach appears intractible for more complicated schemes (even for BDF3), and recurrence relations for the basis functions are given in \cite[Sec.~3.2]{MoScSt2011}.  They can be compactly derived by formally differentiating the generating function \eq{defphi}
with respect to $\xi$ to get
\[%\label{eq:phirecurdef}
    \sum_{j=1}^\infty j \phi_j(t)\, \xi^{j-1} + t \delta'(\xi)
   \sum_{j=0}^\infty \phi_j(t)\, \xi^j = 0\,,
\]
and then collecting terms in $\xi$.  The initial conditions are $\phi_n(t) \equiv 0$ for $n < 0$, and the
first term of the Taylor expansion of \eq{defphi}
gives $\phi_0(t) = e^{-\delta(0) t} = e^{-\delta_0 t}$.
Recurrence relations for BDF1--4 and the trapezoidal rule are given in Table \ref{tab:recur}.

%------------------------------------------- beg table -------------------------------------------
\begin{table}[t]
\begin{center}
\begin{tabular}{|l|c|c|}
\hline
Scheme & Initial & Recurrence for basis functions \\
       & $\phi_{-n} \equiv 0, n\ge 1$ & $j \ge 1$ \\
\hline
BDF1 &$\displaystyle \phi_0(t) = e^{-t}$ &$\displaystyle j\phi_j(t) - t \phi_{j-1}(t) = 0$ \\
\hline
BDF2 &$\displaystyle \phi_0(t) = e^{-3t/2}$ &$\displaystyle j\phi_j(t) - 2t \phi_{j-1}(t) + t\phi_{j-2}(t) = 0$ \\
\hline
BDF3 &$\displaystyle \phi_0(t) = e^{-11t/6}$ &$\displaystyle j\phi_j(t) - 3t \phi_{j-1}(t) + 3t\phi_{j-2}(t) - t\phi_{j-3}(t)= 0$ \\
\hline
BDF4 &$\displaystyle \phi_0(t) = e^{-25t/12}$ &$\displaystyle j\phi_j(t) - 4t \phi_{j-1}(t) + 6t\phi_{j-2}(t) - 4t\phi_{j-3}(t) + t\phi_{j-4}(t)= 0$ \\
\hline
\hline
Trap.\ rule\ & $\displaystyle \phi_0(t) = e^{-2t}$ & $\displaystyle j\phi_j(t) - 4 t \phi_{j-1}(t) + 2(j-1) \phi_{j-1}(t) + (j-2) \phi_{j-2}(t) = 0$ \\
\hline
\end{tabular}
\caption{\label{tab:recur} Recurrence relations for the CQ basis functions.}
\end{center}
\end{table}
%------------------------------------------- end table -------------------------------------------

%--------------------------------------------------------------------------------
\begin{figure}
\begin{center}
\includegraphics[width=0.8\textwidth]{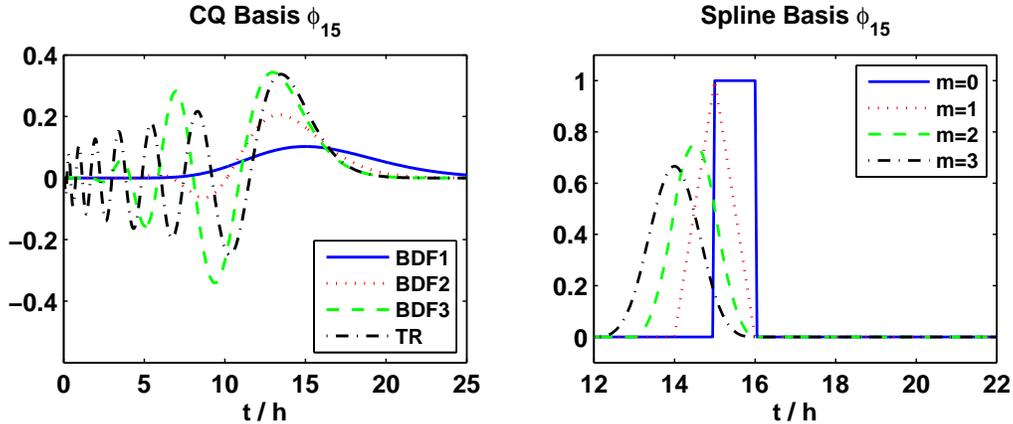}
\end{center}
\caption{\label{fig:basiscompare} Typical CQ and spline basis functions.  See Sections \ref{sec:CQ.eg} and \ref{sec:bspline.props} for details.}
\end{figure}
%--------------------------------------------------------------------------------

%--------------------------------------------------------------------------------
\begin{figure}[h]
\begin{center}
\includegraphics[width=0.8\textwidth]{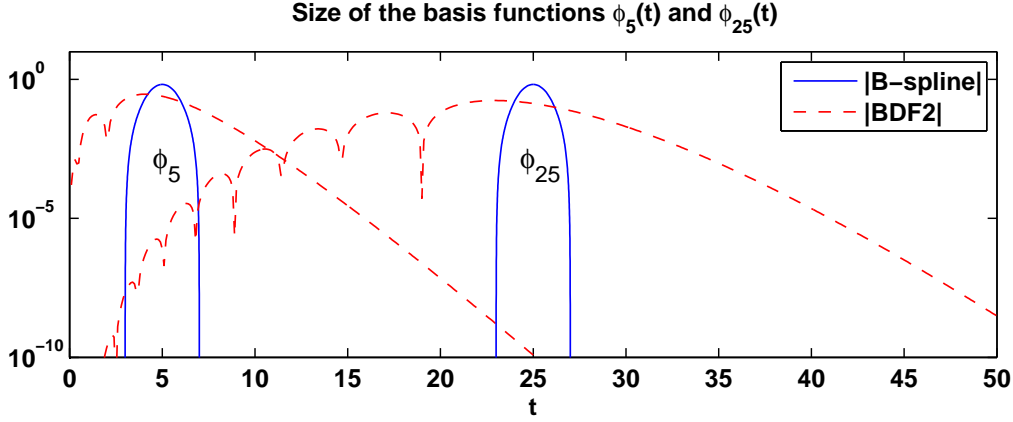}
\end{center}
\caption{\label{fig:basis_BBDF2} Basis functions $\phi_5(t)$ and $\phi_{25}(t)$ for CQ (BDF2) compared with cubic splines (from Sec.~\ref{sec:modbs}).  The effective support of the CQ basis functions $\phi_j(t)$ increases with $j$.}
\end{figure}
%--------------------------------------------------------------------------------

%=====================================================================
%=====================================================================
\newsec{Convolution spline approach}
\label{sec:spline}

As discussed in Sec.\ \ref{sec:intro}, basis functions with global support (such as those described above) give rise to dense matrices $\tens{Q}^j$ in the TDBIE scheme \eq{BIEscheme}, and this has storage and computational cost implications.  Here we explore the use of compactly supported basis functions, which although not derived via standard CQ, nevertheless do fit into the CQ form \eq{CQapprox}.  We set up a general framework for basis functions for the VIE \eq{vie} which are (mainly) translates, and consider specific examples based on B-splines in Secs.\ \ref{sec:bspline}--\ref{sec:modbs}.  This new approach gives sparse system matrices when used to time-step TDBIEs, and results are presented in Sec.~\ref{sec:tdbie}.  It also provides the underpinning theoretical framework for the TDBIE time-stepping approximations of \cite{cq2}.

%--------------------------------------------------------------
\subsection{Construction of a convolution spline scheme for \eq{vie}}

We consider approximations of the form
\eq{CQapprox}, but where all the basis functions $\phi_j$ have {compact support}
of width ${\cal O}(\dt)$ and almost all
are {translates}
of a standard, compactly supported basis function $\phi_m$, i.e.\
\be
  \phi_j(t/\dt) = \phi_m(t/\dt+m -j)\quad \mbox{for $j \ge m$.}
\label{eq:translate}
\ee
When the basis functions are splines, then $m$ is also equal to the polynomial degree. 

Property \eq{translate} means that the approximation $\U(t_n - t) \approx u(t_n - t)$ has the form
 \be
  \U(t_n - t) = \sum_{j=0}^{m-1} \vv_{n-j}\, \phi_j\left(\frac{t}{\dt}\right) + \sum_{j=m}^n \vv_{n-j}\, \phi_m\left(\frac{t}{\dt}+m-j\right)
 \label{eq:cspline}
 \ee
for $t \ge 0$, where $\vv_j$ approximates $u(t)$ for $t$ near
(but not necessarily at) $t_j$, and a sum is defined to be zero if
its upper index is less than its lower index.
Note that when \textbf{all} the $\phi_j$ are translates (as happens for piecewise constant or linear approximations) then $\phi_{n-k}(n-s) \equiv \phi_k(s)$
and the convolution-in-time representation \eq{CQapprox} fits into a standard finite element framework.

Substituting the approximation \eq{cspline} into the integral equation \eq{vie}
and collocating at each time level as described in Section \ref{sec:CQ.ex} gives
\be
  \sum_{j=0}^n \q_j\, \vv_{n-j} = a(t_n)
\label{eq:CSscheme}
 \ee
for $n = 0:N$ where the weights $\q_j$ are defined by \eq{wtdef}.
The unknown coefficients $\left\{\vv_j\right\}_{j=0}^N$ are
then found by time marching as in \eq{timemarching}.  An alternative expression which is useful for analysis is
$\q_0\,\vv_n = \sum_{j=0}^n p_j\, a(t_{n-j})$ for $n \ge 1$,
where the stability coefficients $p_n$ are defined recursively by
\be
  p_0 = 1\,, \quad p_n = \frac{-1}{\q_0}\, \sum_{j=1}^n \q_j\,
  p_{n-j} \quad \mbox{for $n \ge 1$}\,.
\label{eq:defpn}
\ee

%--------------------------------------------------------------
\subsection{Stability of \eq{CSscheme}}

For TDBIE applications and analysis
(see e.g.\ \cite{DaDu2004}), we require the scheme \eq{CSscheme} to be stable in the following sense, independent of the input function $a(t)$.

%----------------------------------------------------------------------
\begin{defn}[Stability] \label{def:stability}
The scheme \eq{CSscheme} is said to be stable when the impulse response
sequence $\left\{p_n\right\}$ defined by \eq{defpn}
satisfies $|p_n| \le C$ for all $n$ such that $n\dt \le T$,
where the constant $C$ is independent of $\dt$.
\end{defn}

This is weaker than BIBO (bounded input bounded output) stability
in the signal processing literature (see e.g.\ \cite{proakis}), which requires boundedness of the
absolute sum $\sum_{n=0}^\infty |p_n| < \infty$.

Stability properties of the scheme \eq{CSscheme} can be established by using the Z-transform, defined as follows.

\begin{defn}%[Z-transform]
The Z-transform of a sequence $\left\{f_n\right\}_{n=0}^\infty$ is the function $F$ given by
\be
  F(\xi) = \Ztrans{f_n}(\xi) = \sum_{n=0}^\infty f_n\, \xi^n
\label{eq:defZtrans}
\ee
where $\xi \in \C$ with $|\xi| \le 1$ is such that the sum converges.
\label{def:Ztrans}\end{defn}

The scheme \eq{CSscheme} is a convolution sum and its Z-transform is
 \be
  \Q(\xi)\, \V(\xi) = A(\xi)\,,
\label{eq:zt1}
 \ee
where
\be
  \Q(\xi) = \sum_{j=0}^\infty \xi^j \int_0^\infty K(t)\, \phi_j(t/\dt)\, dt
\label{eq:ZtransQ}
\ee
and we take $a_n = a(t_n)$.  The $p_n$ coefficients satisfy $\sum_{j=0}^n q_j\, p_{n-j} = 0$ for $n\ge 1$, and when $n=0$ this ``sum'' is equal to $q_0$ (because $p_0 = 1$), and so the Z-transform of \eq{defpn} is $Q(\xi)\, P(\xi) = q_0$, giving
$P(\xi) = \q_0/Q(\xi)\,$.
We now state a sufficient condition for stability when
$Q(\xi)$ is a rational function.

\begin{theorem}[Root condition for stability] \label{def:rootcond}
If the Z-transform $Q(\xi)$ of $\{q_n\}$ is a rational function in $\xi$, then
the approximation \eq{CSscheme} is stable in the sense of Definition \ref{def:stability}
if the roots $\xi_k$ of $\Q(\xi)$ satisfy the following for any constant $c \ge 0$ (independent of $\dt$):
$|\xi_k| \ge 1/(1+c\dt)$ and any with $1/(1+c\dt) \le |\xi_k| \le 1$ are simple.
\end{theorem}

Simple roots with $|\xi_k| = 1/(1+c\dt)$ make a bounded contribution to $p_n$ as $n$ increases by the standard result
\[
              |\xi_k|^{-n} = (1+c\dt)^n \le e^{cT}
\]
for $t_n \le T$, but roots of this size with multiplicity $\mu \ge 2$ contribute terms which grow like $n^{\mu-1}$ and hence violate the stability definition.\smallskip

\textbf{Remark:}\
Although this result is a variant of the \textsl{root condition} familiar (after the change of variable $z = 1/\xi$) from zero stability analysis of numerical methods for ODEs, we note that it does not appear to have previously been derived or used to determine the stability of VIE schemes.\smallskip

Verifying the stability condition
directly or via the root condition above for a general approximation scheme
for \eq{vie} may be very complicated.
But as we show below, schemes with the translate
property \eq{translate} can be tackled within the framework of
Laplace transforms originally introduced for CQ, and this approach
gives a way to extend the scope of stability analysis to a far
broader range of kernel functions.

Substituting the Laplace inversion formula for $K$ into
\eq{wtdef} gives $\ds{\q_j =  \frac{\dt}{2\pi i} \int_{\gamma} \lap{K}(s)\, \Phi_j(-s\dt)\,
  ds}$\,,
where $\Phi_j(s)$ is the Laplace transform of $\phi_j$.  Hence the approximation scheme \eq{CSscheme} can be written as
\be
  a(t_n) = \frac{1}{2\pi i} \int_{\gamma} \lap{K}(s)\, y_n(s\dt)\,ds   \label{eq:yie}
\ee
where\ \ $\ds{y_n(s\dt) =   \dt \sum_{j=0}^n \vv_{n-j} \Phi_j(-s\dt)}$\,.
We note that $y_n$ plays the same role here that the approximate solution of the ODE \eq{ode} does in standard CQ.

The translate property \eq{translate} and the compact support of $\phi_m$ implies
\[
  \Phi_j(-s\dt) = e^{s\dt(j-m)}\, \Phi_m(-s\dt) \quad \mbox{for $j \ge m$}\,
%\label{eq:LTtrans}
\]
and so
\[
  y_n(s\dt) - e^{s\dt}\, y_{n-1}(s\dt) = \dt\, \vv_n\, \Phi_0(-s\dt) +
  \dt \sum_{j=1}^{m} \vv_{n-j}\, \left( \Phi_j(-s\dt) -e^{s\dt}\,\Phi_{j-1}(-s\dt)\right)\,,
\]
(using $\vv_j\equiv 0, j \le 0$).  Taking the Z-transform of this expression gives\ \
$\ds{Y(\xi,s\dt) = \dt\, \symb(\xi,sh)\,\V(\xi)/(1-e^{sh}\xi)}$\ \
when $\xi \ne e^{-s\dt}$, where
\be
  \symb(\xi,s\dt) = \Phi_0(-s\dt) + \sum_{j=1}^{m}
  \left[\st\Phi_j(-s\dt) -e^{s\dt}\,\Phi_{j-1}(-s\dt)\right]\, \xi^j\,.
\label{eq:defgam}
\ee
It hence follows from \eq{yie} that
\[
  A(\xi) = \V(\xi)\frac{\dt}{2\pi i} \int_{\gamma} \lap{K}(s)\, \left(\frac{\symb(\xi,sh)}{1-e^{sh}\xi}\right)\, ds
\]
and comparison with \eq{zt1} yields the alternative representation for the Z-transform of the weights $\q_j$:
\be\label{eq:Qdef}
     \Q(\xi) = \frac{\dt }{2\pi i} \int_{\gamma} \lap{K}(s)\, \left(\frac{\symb(\xi,sh)}{1-e^{sh}\xi}\right)\, ds\,.
\ee

The expression $\symb(\xi,s\dt)/(1-e^{s\dt}\,\xi)$ plays a role similar to that of $\left(\delta(\xi)/\dt - s\right)^{-1}$ in standard CQ analysis, and it is the key quantity in determining whether the scheme is stable or not.
Unfortunately it has a more complicated structure: it
has an infinite vertical line of simple poles at $s = \pole_k$ for
$k \in \mathbb{Z}$, where
\be
\label{eq:poles}
  \pole_k := \frac{1}{\dt}\left(-\ln |\xi| -i\,\mathrm{Arg}(\xi) + i\, 2\pi k\right)
\ee
and the principal argument $\mathrm{Arg}(\xi) \in (-\pi,\pi]$.
Note that if $|\xi|<1$ then $\mathrm{Re}(\pole_k)>0$.

%---------------------- begfig ----------------------------------
\begin{figure}[t]
\begin{center}
\includegraphics[width=3.2in]{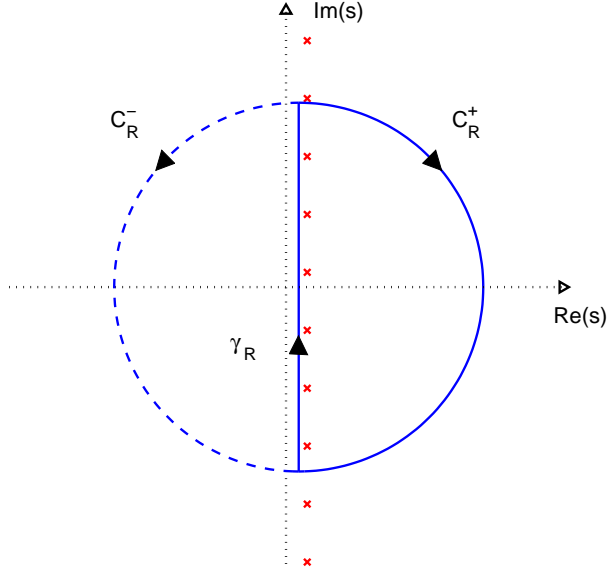}
\end{center}
\caption{\label{fig:contours} The left and right D-contours of radius $R$ used for the stability and Z-transform calculations. The crosses
are the poles \eq{poles} and the vertical line of length (approximately) $2R$ is
$\gamma_R$.}
\end{figure}
%---------------------- begfig ----------------------------------

To evaluate $Q(\xi)$ defined by \eq{Qdef} for a given kernel
function $K(t)$ we can use either the left or right D-contours
illustrated in Figure \ref{fig:contours}, taking the limit
$R\rightarrow \infty$ and setting $\gamma = \lim_{R\rightarrow
\infty} \gamma_R$. Using the right contour gives
\[ %\label{eq:rightcontour}
  Q(\xi) = \sum_{k=-\infty}^\infty
  \lap{K}(\pole_k)\, \symb(\xi,\pole_k\dt) -
  \lim_{R\rightarrow\infty} \frac{\dt}{2\pi i} \int_{C^+_R} \lap{K}(s)\,
  \left(\frac{\symb(\xi,s\dt)}{1-e^{s\dt}\,\xi}\right)\,  ds\,.
\]
The integral round $C^+_R$ does not necessarily vanish as $R
\rightarrow \infty$ since for some basis functions (including higher
order B-splines) the quantity\ \ $\ds{\frac{\symb(\xi,s\dt)}{1-e^{s\dt}\,\xi} = {\cal O}(e^{c s\dt})}$
as Re$(s) \rightarrow \infty$ for $c \ge 1$. We may also use the left contour when $\lap{K}(s)$
has simple poles at $s=\kappa_j$ with $\mathrm{Re}(\kappa_j) \le 0$
and obtain the analogous result
\[
  Q(\xi) = \dt \sum_j \lim_{s\rightarrow\kappa_j} \left(\frac{(s-\kappa_j)\,\symb(\xi,s \dt)}{1-e^{s\dt}\, \xi}\;
    \lap{K}(s)\right) -
  \lim_{R\rightarrow\infty} \frac{\dt}{2\pi i} \int_{C^-_R}
  \lap{K}(s)\, \left(\frac{\symb(\xi,s\dt)}{1-e^{s\dt}\,\xi}\right)\,  ds\,.
\]
The asymptotic behaviour of the integral $C^-_R$ as $R\rightarrow
\infty$ is determined primarily by $\lap{K}(s)$. The extension of
this left contour approach to poles with higher multiplicity is
straightforward.  We
illustrate the use of these formulae in Sec.~\ref{sec:bspline.stab} for various
kernels $K$ when the basis functions are B-splines.

%=====================================================================
%=====================================================================
\newsec{B-spline basis functions for \eq{vie}}
\label{sec:bspline}

We now illustrate the theoretical framework introduced in Section \ref{sec:spline} for basis functions $\phi_j$ which are B-splines on $[0,\infty)$.  We begin by listing some general properties
of B-splines which are needed in the subsequent analysis, and then examine the stability of the convolution spline approximation of \eq{vie} for different example kernels.  We also prove that the approximation given by \eq{CSscheme}
converges to the solution $u$ of \eq{vie} for general smooth $a$ and
$K$.  The convergence rate is at most second order, no
matter how high the polynomial degree, because quasi-interpolation
by the Schoenberg B-spline operator is at most $\oo(\dt^2)$
\cite{deboor}.  However a simple modification of the B-spline basis near $t=0$ can give higher order stable approximations of \eq{vie}, and this is analysed for the cubic case in Section~\ref{sec:modbs}.

%=====================================================================
\subsection{Notation and properties}
\label{sec:bspline.props}

We now look in detail at the approximation \eq{cspline} when 
the basis functions are (iso-geometric)
B-splines of polynomial degree $m$ based on the uniformly spaced nodes (or knots)
$t_j = j\, \dt$ for $j \ge 0$. It is necessary for the B-spline basis functions to have the sum to unity
property \eq{pou} in the whole interval $[0, \infty)$, and we introduce $m$ new knots $t_j = 0$
for $j = -m:-1$. The $m$th degree B-splines are $b^m_j(t)$ for $j
\ge -m$, and B-splines of degree $m > 0$ are recursively defined in
terms of those of lower degree as follows, using the convention that
$b^m_j(t) \equiv 0$ for $j < -m$.

\begin{defn}\cite{deboor}
When $m = 0$
\[
  b^0_j(t) = \left\{\begin{array}{cl}
    1 &\quad\mbox{if $t \in[t_j, t_{j+1})$ for $j \ge 0$ and}\\
    0 &\quad\mbox{otherwise.}
  \end{array}\right.
\]
If $m > 0$ then
\[
  b^m_j(t) = \left(\frac{t - t_j}{t_{j+m} - t_j}\right)\, b^{m-1}_j(t) + \left(\frac{t_{j+m+1} - t}{t_{j+m+1} - t_{j+1}}\right)\, b^{m-1}_{j+1}(t)
\]
where the convention is that $0/0$ is interpreted as 0.
\end{defn}

Throughout this section we shall use basis functions
\be\label{eq:phibs}
  \phi_j(t/\dt) = b^m_{j-m}(t)\quad \mbox{for $j \ge 0$}.
\ee
Note that the spline degree $m$ is also the translate parameter from \eq{translate}.

We make use of several B-spline properties in Section
\ref{sec:bspline} (see for example
standard references such as \cite{deboor,schumaker}), which we list here for convenience.
\bigskip

\textbf{B-spline properties} %\cite{bspline}
\begin{enumerate}
\item Compact support.\quad  $b^m_j(t) = 0$ outwith $[t_j, t_{j+m+1})$, and $b^m_j(t_j) = 0$ unless $j = -m$.
\label{bsprop:supp}

\item Translate property. \label{bsprop:trans}
If $j \ge 0$ then $b^m_j(t) = b^m(t/\dt-j)$, where the functions $b^m$ are defined recursively:
\[
  b^0(\tau) = \left\{\begin{array}{cl}
  1 &\ \mbox{if $\tau \in[0, 1)$,}\\
    0 &\ \mbox{otherwise}
  \end{array}\right.
  \quad\mbox{and if $m \ge 1$:}\quad
  b^m(\tau) = \frac{\tau}{m}\, b^{m-1}(\tau) + \frac{m+1-\tau}{m}\, b^{m-1}(\tau-1)\,.
\]
It follows that $\phi_j(\tau) = b^m(\tau + m - j)$ for $j \ge m$.

\item Sum to unity.\quad $\ds{\sum_{j=-m}^\infty b^m_j(t) = 1}$ for all $t \ge 0$.
\label{bsprop:pou}

\item Moments.
\[
  \int_{t_{j-m}}^{t_{j+1}} b_{j-m}^m(t)\, dt = \frac{t_{j+1} -
  t_{j-m}}{m+1}\quad \mbox{and}\quad \int_{t_{j-m}}^{t_{j+1}} t\, b_{j-m}^m(t)\, dt = \frac{t_{j+1} -
  t_{j-m}}{(m+1)(m+2)}\; \sum_{k=0}^{m+1} t_{j-m+k}\,.
\]
\label{bsprop:mom}

\item Shoenberg quasi-interpolation.  Suppose that $m \ge 1$ and set $\ds{t^m_{j} = \frac{\dt\,(m+j)\,(m+j+1)}{2 m}}$ for $j = -m:-1$ and $t^m_j = t_{j +(m+1)/2}$ for $j \ge 0$.  Then\ \
$\ds{\sum_{j=-m}^\infty t^m_j\, b^m_j(t) = t}$ when $t \ge 0$.
\label{bsprop:qi}
\end{enumerate}

It follows from properties \ref{bsprop:supp}, \ref{bsprop:pou} and
\ref{bsprop:qi} above that
 \be
  f(t) = \sum_{j=-m}^\infty f(t_j^m)\, b^m_j(t) +  \oo(\dt^2)
\label{eq:QIresult2} \ee for any $f \in C^2[0,\infty)$, and if $f
\in C^{p+1}[0,\infty)$ for $p \ge 2$ and $t \in [t_\ell,
t_{\ell+1})$ for some $\ell \ge 0$, then
\be%gin{eqnarray}
  f(t) - \sum_{j=\ell-m}^\ell f(t_j^m)\, b^m_j(t) =\sum_{k=2}^p
  \frac{f^{(k)}(t_\ell)}{k!}\, \left[ (t-t_\ell)^k - \sum_{j=\ell-m}^\ell
  \left(t_j^m - t_\ell\right)^k\, b^m_j(t)\right] + \oo(\dt^{p+1})\,.
  \label{eq:QIresult4}
 \ee%nd{eqnarray}

It follows from \ref{bsprop:supp} that the CQ weights are
 \be
  \q_j = \int_{t_{j-m}}^{t_{j+1}} K(t)\, b_{j-m}^m(t)\, dt\,.
\label{eq:defbw}
 \ee
The convergence analysis relies crucially on
knowing the values of the weights when $K$ is a constant, and
this follows immediately from \ref{bsprop:mom}: when $K \equiv 1$
the weights $\q_j$ of \eq{defbw} are given by
 \[
  \frac{\q_j}{\dt} = \left\{\begin{array}{ll}
    \frac{j+1}{m+1} & \quad\mbox{for $j = 0:m-1$}\\
    1 & \quad\mbox{if $j \ge m$\,.}
  \end{array}\right.
%\label{eq:wval}
 \]

%=====================================================================
\subsection{Stability results for convolution B-splines}
\label{sec:bspline.stab}

We now use the theoretical framework introduced in Section \ref{sec:spline} to examine the stability of the convolution B-spline approximation of \eq{vie} for different example kernels which capture some of the important properties of TDBIE problems.
These are: $K(t)$ equal to a constant, a step function, and the highly oscillatory kernels $K(t) = J_0(\w t)$ or $\cos(\w t)$, where $\w$ can be of the order of $1/\dt$.  We use $\symb_m(\xi,s\dt)$ to denote the function defined by \eq{defgam} for the degree $m$ basis functions, and $\Q_m(\xi)$ to denote the coefficient Z-transform given by \eq{Qdef}.  The first few values of $\symb_m(\xi,s\dt)$ are listed in Table \ref{tab:gamma}; those for higher values of $m$ are more complicated, but are easily computed in a standard algebraic manipulation package.

In three of the cases $Q_m(\xi)$ is a rational function in $\xi$ and Theorem \ref{def:rootcond} can be used to determine stability.  The Bessel function case is more complicated, and stability is determined from the Z-transform inversion formula by bounding the coefficients $p_n$ of \eq{defpn} directly.  Note that this bound is independent of $n$, and so is a practically useful stability result, in contrast with the
(essentially) uncheckable hypotheses needed in \cite{DaDu2003}.

%-------------------------------------------------------------------------
\begin{table}[ht]
\[
\begin{array}{|c|c|c|c|}
\hline
\mathrm{B\!\!-\!\!spline} & \symb_m(\xi,s) & \lim_{s\rightarrow 0} \symb_m(\xi,s) & \symb_m(e^{-s},s)\\
\mathrm{degree} & & &\\
\hline
     m=0 & s^{-1}(e^s-1) & 1 & s^{-1}(e^s-1)\\
\hline
     m=1 & s^{-2}\left[e^s(\xi s-\xi+1) + (\xi-s-1)\right] & (1+\xi)/2
           & s^{-2}(e^s-1)^2 e^{-s}\\
\hline
     m=2 & s^{-3}\left[e^{2s}(\xi^2-\xi)+ e^s(2-2\xi^2s+ \xi(s^2+2s-2)) + \right.& & \\
         &       \left. ((2s+3)\xi-s^2-2s-2-\xi^2)\right] & (1+\xi+\xi^2)/3
         & s^{-3}(e^s-1)^3 e^{-2s}\\
\hline
   m=3 & s^{-4} (e^s\xi -1)  \left[ s^3-3\,s^2 ( \xi-
1 ) +3\,s ( \xi-1 )  ( \xi-2 ) \right. & & \\
   &-\frac{1}{2}\, ( \xi
-1 )  ( 2\,\xi^2-9\,\xi+12 ) & & \\
 & +\left. \frac{1}{2} e^s ( \xi-1 )
   ( e^s\xi-6\,\xi+10 ) \right] &  (1+\xi+\xi^2+\xi^4)/4 & s^{-4}(e^s-1)^4 e^{-3s} \\ 
 & + s^{-4}e^s(\xi-1 ) ^4 & & \\
\hline
\end{array}
\]
\caption{\label{tab:gamma} The function $\symb_m(\xi,s)$ for $m=0:3$.  See text for details.}
\end{table}
%-------------------------------------------------------------------------

%-------------------------------------------------------------------------
\subsubsection{Constant kernel: $K(t)=1$, transform $\bar{K}(s)=1/s$} \label{sec:constker}

Integrating \eq{Qdef} round the left contour in Figure \ref{fig:contours}
gives
\be \label{eq:ker1}
    \Q_m(\xi) = \dt \lim_{s\rightarrow 0} \frac{\symb_m(\xi,s\dt)}{1-e^{s\dt} \xi}
     = \frac{\dt(1-\xi^{m+1})}{(m+1)(1-\xi)^2}\,.
\ee
The function $\Q_m$ has $m$ simple roots on the unit circle, and stability of the approximation then
follows from Theorem \ref{def:rootcond}.  (Note that stability also follows from the convergence result of Sec.\ \ref{sec:bspline.conv}.)

%-------------------------------------------------------------------------
\subsubsection{Discontinuous step-function kernel: $K(t)=1$ for $t \in [0,L]$, otherwise $0$.}
\label{sec:BSstab.disc}

Discontinuous kernels can
arise in TDBIE problems, even when the scattering surface $\Gamma$ is smooth and closed.
Examples (in Laplace transformed representation) are given in \cite[Sec.\ 6.1]{BaLu2011} and \cite[Sec.\ 4.1]{SaVe2011a} describing time domain scattering where only the zeroth order harmonic
in space is excited on the surface of a sphere.  Similar, but more
complicated discontinuous kernels are described in \cite{SaVe2011a} for more general
scattering from spheres involving higher spatial harmonics.

We assume that the duration $L$ is
independent of $\dt$ and denote the integer part of $L/h$ by $\Md$,  i.e.\ when $\dt$ is sufficiently small,
$L = (\Md+r)\, \dt$ for integer $\Md > m$ and $r \in [0,1)$\,.
It is simplest to work with the explicit Z-transform formula \eq{defZtrans}
using the weights given in \eq{defbw}.  Results for $m = 0:3$ are summarised below.
\medskip

\textbf{Case} $\mathbf{m=0}$:\quad\quad\quad     $\ds{\frac{Q_0(\xi)}{\dt} = r\, \xi^\Md +  \sum_{n=0}^{\Md-1} \xi^n = \frac{1}{\xi-1}\left(r\, \xi^{\Md+1} + (1-r)\,\xi^\Md - 1\right)\,.}$

When $r\in(0,1)$ it can be shown that the $\Md$ roots $\xi_j$ of $Q_0$
satisfy $|\xi_j|>1$ for $j=1:\Md$, and when $r=0$ there are $\Md-1$ simple roots
$\xi_j = \exp(i 2\pi j/\Md)$ for $j=1:\Md-1$.  Hence Theorem \ref{def:rootcond} implies that the $m=0$ scheme is stable for all $L$.
\medskip

\textbf{Case} $\mathbf{m=1}$: We have
\[
  Q_1(\xi) = 
    \frac{\dt}{2}\, \frac{1}{1-\xi}\left(1+\xi - \xi^\Md g\right),
		\quad \mathrm{with\ }  g(r,\xi) = \xi+(1-r+r\xi)^2 \quad \mathrm{and\ } r \in [0,1).
\]
Using the definition \eq{defpn} and formal power series expansion for small $\xi$ gives
\[
        P_1(\xi) = \frac{\dt}{2 Q_1(\xi)} = \frac{1-\xi}{1+\xi - \xi^M g(r,\xi)}
				         = \frac{1-\xi}{1+\xi}\left(1 + \xi^\Md \frac{g(r,\xi)}{1+\xi} + \xi^{2\Md} \frac{g(r,\xi)^2}{(1+\xi)^2} + \ldots \right) = \sum_{n=0}^\infty p_n \xi^n
\]
where the $p_n$ are the stability coefficients.
The finite duration of the kernel has no impact on the $p_n$ until $n \ge \Md$,
and it is relatively easy to show that in the first time interval after that we have
\[
          p_n = 2 (-1)^n + (-1)^{n+\Md}(2-4r^2-8r(1-r)(n-\Md)),
					\quad \Md+2 \le n \le 2\Md-1.
\]
When $r \in (0,1)$ we have $p_n = \oo(n)=\oo(\dt^{-1})$, and the scheme is unstable by Definition \ref{def:stability}. 
In subsequent time periods (measured in terms of the duration $L$) it can be shown that the instability gets worse and 
$p_n = \oo(n^{\floor{t_n/L}})$.
In the special case when $r=0$ (or equivalently $L = M\dt$) this scheme is stable for this problem, but it may not be possible to satisfy similar integer multiple of $\dt$ conditions in a more complicated problem, for example when there are two or more time periods whose ratios are irrational.

\textbf{Case} $\mathbf{m=2,3}$: A similar argument can be used to show that these two schemes are unstable for all $r \in [0,1]$, and that in each case $p_n = \oo(n^{\floor{t_n/L}})$.
Note however that the modified cubic spline basis functions described in Section \ref{sec:modbs} give completely stable results for this kernel.

%-------------------------------------------------------------------------
\subsubsection{$K(t) = J_0(\w t)$, transform $\bar{K}(s) = 1/\sqrt{s^2+\w^2}$}
\label{sec:BSstab.bessel}

This is the kernel function that arises when considering TDBIE scattering from the flat surface $\R^2$, where $\w$ can be of the order of $1/\dt$ (i.e.\ $\dt\w$ is bounded as $\dt \rightarrow 0$, but does not necessarily tend to zero).
Its Laplace transform has a branch cut between the values
$s = \pm i \w$, and the Z-transform $Q_m(\xi)$ of the weights is not a rational function.
We can still establish stability directly for the
impulse response sequence $\{p_n\}$ defined in \eq{defpn}  using a change of variable in
the Z-transform inversion formula \cite[eq.\ 37.7]{doetsch}
to get
\be \label{eq:pnbyinv}
         p_n = \frac{e^{n\dt \sigma}q_0}{2\pi} \int_{-\pi}^{\pi} \frac{e^{i n \y}}{Q_m(e^{-\x-i\y})}\; d\y\,,
\ee
where we have set $\xi = e^{-s\dt}$ with $s = \sigma + i \eta$ and $\sigma > 0$
and then changed to scaled variables $\x=\sigma\dt$ and $\y = \eta\dt$.  This yields the bound
\be \label{eq:pnbnd}
         |p_n| \le \frac{e^{\sigma T}}{2\pi} \int_{-\pi}^{\pi} \frac{|q_0|\, d\y}{|Q_m(e^{-\x-i\y})|}
\ee
when $t_n \le T$, which
holds for any fixed $\sigma > 0$ when the singularities of the integrand are to the left of $\x$. Note that this bound is independent of $n$, and the scheme is stable at a given frequency $\w$ if the integral term in \eq{pnbyinv} remains bounded as $\dt \rightarrow 0$.  This can be demonstrated using the right contour in Fig.~\ref{fig:contours} to calculate $Q_m(e^{-s\dt})$ but it is more straightforward to work directly with \eq{ZtransQ}.

It follows from standard properties of the B-spline basis \eq{phibs} that
\[
  q_0 = \int_0^\dt \left(1- \frac{t}{\dt}\right)^m J_0(\w\, t)\, dt = \frac{m!\,\intK{m-1}(\dt)}{\dt^m}\,
\]
where functions $\intK{k}(t)$ are recursively defined by
\[ %\label{eq:intKdef}
     K^{(0)}(t)= J_0(\w\,t), \quad  \intK{k-1}(t) = \int_0^t  \intK{k}(t')dt' \quad
     \mathrm{for\ } k=0,1,\ldots.
\]
Note that
\be \label{eq:intKbnd}
    \intK{m-1}(t) \le t^{m+1}/(m+1)!
%    \quad\mathrm{and}\quad \intK{m}(t)  = \oo(\w^{-m})
\ee
for all $t\ge 0$. %and large $\w$.

Properties of the B-spline basis functions can also be exploited to write \eq{ZtransQ} as
\be \label{eq:Qmhardway}
   Q_m(\xi) = \frac{(1-\xi)^{m+1}}{\xi\, \dt^m}\; \Ztrans{\intK{m-1}}(\xi) +  C_m(\xi)
\ee
where the correction terms are $C_0 = 0, \quad C_1 = 0$,
\[
        C_2(\xi) = (1-\xi) \frac{\intK{3}(\dt)}{\dt^2},
     \quad
      C_3(\xi) = (1-\xi)(5-3\xi) \frac{\intK{4}(\dt)}{\dt^3}
                     + \xi(1-\xi)\frac{\intK{4}(2\dt)}{2\, \dt^3}\;.
\]
The presence of these terms is because for $m=0:1$ the basis functions are pure translates, while for $m \ge 2$, there are  different shaped basis functions at the start.
The function $\intK{k}$ has Laplace transform
\[
  \lap{\intK{k}}(s) = \frac{1}{s^k \sqrt{s^2+\w^2}}
\]
and it follows from the Poisson sum formula relating $Z$ and Laplace transforms that
\[
  \Ztrans{\intK{k}}(e^{-s\dt}) = \frac{1}{\dt} \sum_{j\in\Z} \frac{1}{s_j^k \sqrt{s_j^2 + \w^2}} := \dt^k \sum_{j \in \Z} f_j^{k-1}\,
\]
where $s_j = s + i 2\pi j/\dt$,
and we use this expression in \eq{Qmhardway} in order to bound the integral term in \eq{pnbnd}.

When $m=0$ it is possible to obtain an analytic bound when $\w \le \pi/\dt$, and a careful numerical approximation of the integral \eq{pnbnd} indicates that the $p_n$ are bounded for $\w$ up to (at least) $20\,\pi/\dt$.  The situation is more complicated for $m=1:3$, and in these cases we give numerical bounds.
\medskip

%-----------------------------------

\textbf{Case} $\mathbf{m=0}$:
From \eq{intKbnd} and \eq{Qmhardway}
\[
         \frac{|q_0|}{|Q_0(e^{-\x-i\y})|}
         = \frac{1}{|e^{x+iy} -1|\,\left|\sum_{k\in\Z} f_k^0\right|}
\]
and
\[
  \left|\sum_{k\in\Z} f_k^0 \right|
  = |f_0^0| \left(\left(1+\sum_{k\in\Z/0}\Re(f_k^0/f_0^0)\right)^2 + \left(\sum_{k\in\Z/0}\Im(f_k^0/f_0^0)\right)^2\right)^{1/2}
  \ge |f_0^0| \left|1+\sum_{k\in\Z/0}\Re(f_k^0/f_0^0)\right|\,.
\]
It can be shown that
\[
       1+\min_y\sum_{k\in\Z/0}\Re\left(f_k^0/f_0^0\right)
       =
       1+\sum_{k\in\Z/0}\Re\left(f_k^0/f_0^0\right)|_{y=0}
       > \frac{2}{3}
\]
when $0 \le \x \le 1$, $0 \le \w\dt \le \pi$ and $|y| \le \pi$.
In this case
\[
  \frac{|q_0|}{|Q_0(e^{-\x-i\y})|} \le
  \frac{3\, \sqrt{\x^2+ 2\pi^2}}{2}\; \frac{\sqrt{\x^2+\y^2}}{|e^{x+iy}-1|} \le
  \frac{3\, \sqrt{\x^2+ 2\pi^2}}{2}\; \frac{\pi\, e^{-x/2}}{2}
\]
using Jordan's inequality.
Together with \eq{pnbnd} this proves that the
scheme is stable in the sense of
Definition \ref{def:stability} for frequency $\w$
in the contiguous interval $0 \le \w\dt \le \pi$.
Numerical evaluation of the right hand side of \eq{pnbnd} indicates that the bound is
\[
    |p_n| \le 1.3\, e^{\sigma T}
\]
when $\dt$ is sufficiently small (so that $\x < 0.1$) and $0 \le \w\dt \le 20\pi$.
Further numerical tests computing $p_n$ directly from \eq{defpn} for a finite number of steps $n \le 2500$ and the same range of values of $\w\dt$ indicate that $|p_n| \le 1$, consistent with the
estimate above.
There is no indication of instability at any value of $\w\dt$ tested
and we speculate that this scheme is stable for all $\w$.
\medskip

%-----------------------------------
\textbf{Case} $\mathbf{m=1}$:
Finding an explicit bound for the integral in \eq{pnbnd}
is significantly more complicated and perhaps even intractible here
so we only consider its direct numerical evaluation
over a range of frequencies and values of $\x = \dt\sigma$ close to 0.
However there is an extra complication because $q_0/Q_1(z)$ has a
pole at $z=-1$.
This is most obvious when we set $\w=0$ and get
$q_0/Q_1(z) = (1-z)/(1+z)$ from \eq{ker1}.
When $\w\ne 0$ there is no simple formula, but it is still possible to show
by direct evaluation of the summation formula for $Q_1(e^{-\x-i\y})$
that the pole remains when $0<\w\dt<\pi$.
The pole renders the bound in \eq{pnbnd} less useful since
\[
     |q_0|\int_{-\pi}^{\pi} \frac{dy}{|Q(e^{-\x-i\y})|} = \oo(\log(1/\x)) \rightarrow \infty
     \quad \mathrm{as\ } \x \rightarrow 0
\]
(where $\x = \dt\sigma$) and hence $|p_n| \le (C_0 + C_1\log(1/\dt)) e^{\sigma T}$
as $\dt \rightarrow 0$, which
does not satisfy the stability requirement of Def.~\ref{def:stability}.
Fortunately the singularity can be removed by writing
\[
     \frac{q_0}{Q_1(\xi)} = \frac{a}{(1+\xi)} + \Delta P(\xi),
     \quad
     \mathrm{where\ }
     a = \lim_{\xi \rightarrow -1} \frac{q_0(1+\xi)}{Q_1(\xi)}
\]
so that $\Delta P(\xi)$ is bounded as $\xi \rightarrow -1$.
The sequence $\{p_n\}$ can then be written as $p_n = a\,(-1)^n + \Delta p_n$
where $\Delta p_n$ is bounded in the same way as \eq{pnbnd}:
\be \label{eq:deltapnbnd}
       |\Delta p_n| \le \frac{1}{2\pi} \int_{-\pi}^{\pi} |\Delta P(\x+i\y)|\, dy.
\ee
Numerical evaluation of the integral over frequencies $0 \le \w\dt < \pi$ indicates that
$|\Delta p_n| \le 1.1$ and $0 < a \le 2$
for $n\dt \le T$ and $0 < \x \le 1/10$.
Combining this with direct evaluation of \eq{pnbnd} when $\w\dt \in [\pi,20\pi]$ and there is not a pole indicates that
\[
        |p_n| \le C e^{\sigma T}
        \quad \mathrm{where} \quad C = \left\{\begin{array}{ll}
                  3.1\,,  &\ \w\dt \in [0,\pi)\\
                  1.1\,,  &\ \w\dt \in [\pi,20\pi]
                  \end{array}\right.
\]
for $n\dt \le T$ and $0 < \x \le 1/10$ satisfying the stability Def.~\ref{def:stability}.
Further numerical tests computing $p_n$ directly from \eq{defpn} for a finite number of steps $n \le 2500$ and the same range of values of $\w\dt$ indicate that $|p_n| \le 2$ for $\w\dt \in [0,0.7\pi)$ and $|p_n| \le 1$ for $\w\dt \in [0.7\pi,20\pi]$, consistent with the
estimate above.
Again we speculate that this scheme is stable for all $\w$.
\medskip

%-----------------------------------
\textbf{Case} $\mathbf{m=2}$:
The function $q_0/Q_2(\xi)$ appears to have two poles on the unit circle
when $\w\dt \in [0,L)$ where $L \approx 2.55$, symmetrically located at $\xi = e^{\pm i \mu(\w\dt)}$. In the simple case $\w=0$, \eq{ker1} gives  $Q_2(\xi) = q_0\,(1+\xi+\xi^2)/(1-\xi)\,$,
and so $\mu(0) = 2\pi/3$.
Numerical evidence indicates that $\mu(\w\dt)> \w\dt$ and that $\mu$ increases until the two poles meet where $\mu(L)=\pi$.
At that point stability in the sense of Def.~\ref{def:stability} breaks down since there does not appear to be any compensating factor in the numerator to reduce the order of this double singularity.

We locate the poles numerically, and remove them from the integrand $q_0/Q_2(\xi)$ in a similar way to the previous case. The simplest form that captures the main features of the behaviour is
\[
   \frac{q_0}{Q_2(\xi)} = \frac{a(1-\xi)}{\xi^2-2\xi\cos\mu+1} + \Delta P(\xi),
\]
so that by direct inversion of the Z transform
\[
    p_n = a (\cos(n\mu) - \sin(n\mu) \tan(\mu/2)) + \Delta p_n.
\]
For $0 \le \w\dt < L \approx 2.55$ we find that
$0 < a \le 1$ and from \eq{deltapnbnd} that $|\Delta p_n| \le 0.8$, giving
\[
     |p_n| \le 0.8 + \sec(\mu(\w\dt)/2)
\]
for $n\dt \le T$ when $0 < \x \le 1/10$.
This satisfies Def.~\ref{def:stability} since $2\pi/3 \le \mu(\w\dt) < \pi$,
but since $\sec(\mu/2) \rightarrow \infty$ as $\mu \rightarrow \pi$, the possibility for instability is clear.
Further numerical tests computing $p_n$ directly from \eq{defpn} for a finite number of steps show very close and consistent agreement with this bound on $|p_n|$, with instability appearing as predicted at $\w\dt = L \approx 2.55$ -- i.e.\ there is
a contiguous interval of stability $\w\dt \in [0,L)$ with $L \approx 2.55$.
\medskip

%-----------------------------------
\textbf{Case} $\mathbf{m=3}$:
Not surprisingly this case is more complicated still.
When $\w=0$ the three poles of $q_0/Q_3(\xi)$ are on the unit circle at $\xi = -1, e^{\pm i\pi/2}$.
However, when $\w\dt>0$ increases, the real-valued pole at $\xi=-1$ moves (harmlessly) outside the unit circle
while the other complex conjugate pair moves inside causing instability.
Numerical tests computing $p_n$ directly from \eq{defpn} for fixed vlaues of $\w\dt$
show behaviour consistent with
this: we see apparent stability for larger values of $\dt$ which disappears as
$\dt \rightarrow 0$, i.e.\ this
scheme is stable only when $\w$ is fixed (so that $\w\dt \rightarrow 0$).
\medskip

Similar results can be proved for the (more straightforward) oscillatory kernel $K(t) = \cos \w t$, as summarised below.
\bi
\item $\mathbf{m=0:}$  the scheme is stable at any frequency $\w$ for which $\w\dt \in [0, \pi)$\,;
\item $\mathbf{m=1:}$ the scheme is stable at any frequency $\w$ for which $\w\dt \in [0, 2\pi)$\,;
\item $\mathbf{m=2:}$ the scheme is stable at any frequency $\w$ for which $\w\dt \in [0, \theta)$, where $\theta = 1.9747\dots$\,;
\item $\mathbf{m=3:}$ there is no $\oo(1)$ interval of stability for $\w\dt$, but the scheme is stable for bounded $\w$.
\ei

The stability results for highly oscillatory kernels are illustrated in Figure~\ref{fig:stabJ0andcos}.  The plots show
$\max_n |p_n|$ for $n=0:2500$ for the
B-spline schemes with $m = 0:3$ applied to the kernels $K(t) = J_0(\w t)$ (left plot) and $K(t) = \cos \w t$ (right plot). Over the
range $\omega \in [0,\pi/\dt]$ shown, the general
stability behaviour for these two kernels is similar.
In particular the left plot illustrates the stability when $m=0,1$, while
scheme $m=2$ is stable for $\omega\dt \in [0,L)$ with $L \approx 2.55$.
On the right plot, scheme $m=0$ is stable except at $\omega \dt = \pi$,
scheme $m=1$ is stable and $m=2$ scheme is stable for  $\omega\dt \in
[0,L)$ with $L \approx 1.97$.
On both plots the $m=3$ scheme is clearly unstable when $\omega =
\oo(1/\dt)$.

%--------------------------------------------------------------------------------
\begin{figure}
\begin{center}
\includegraphics[width=0.8\textwidth]{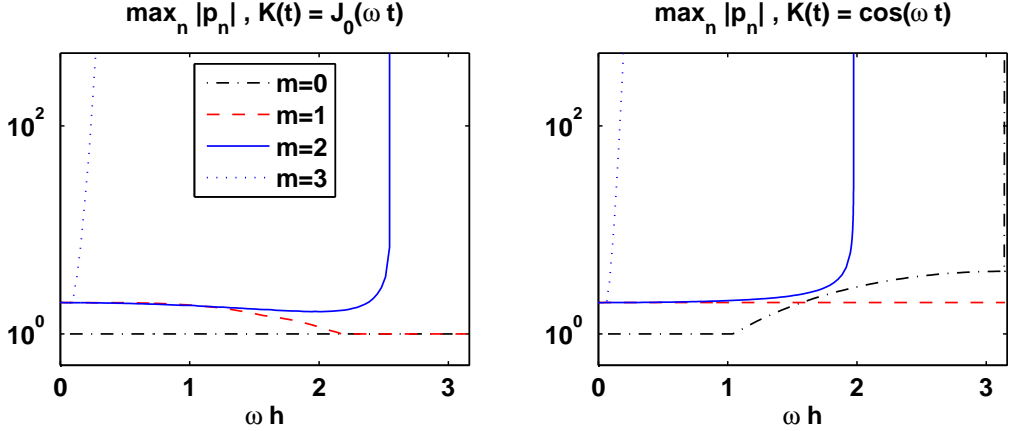}
\end{center}
\caption{\label{fig:stabJ0andcos} Plots of $\max \left\{|p_n|: 0 \le n \le 2500\right\}$ against $\w \dt$ for the
B-spline schemes with $m = 0:3$ applied to the highly oscillatory kernels $K(t) = J_0(\w t)$ (left plot) and $K(t) = \cos \w t$ (right plot).  See text for more details.}
\end{figure}
%--------------------------------------------------------------------------------

%=====================================================================
\subsection{Convergence results for \eq{CSscheme}}
\label{sec:bspline.conv}

Formal convergence of a method when applied to a smooth VIE problem is certainly a necessary condition for it to behave well (for VIEs or TDBIEs), and we now prove that the collocation spline approximation
\eq{CSscheme} with B-spline basis functions \eq{phibs}
converges to the solution $u$ of \eq{vie} for general smooth $a$ and
$K$.  The analysis proceeds by considering the case $K \equiv 1$ and then using Taylor
expansion to show that the same result also holds for smooth $K$ with $K(0) = 1$ when $\dt$ is
small enough (see e.g.\ \cite{brunnerbook}). Thus it does not apply
to the important case of an oscillatory kernel
where the oscillation frequency $\w = \oo(1/\dt)$, whose stability was analysed above for the Bessel function and cosine kernels.

When $m=0$ the approximation \eq{cspline} is the same as
using piecewise constant collocation (at the interval endpoints),
and this has been fully analysed (see e.g.\ \cite{brunnerbook} for
details).  Here we assume that $m \ge 1$ (note that this includes the well-known case of piecewise linear approximations of \eq{vie}), and show that convergence is always second order, no
matter how high the polynomial degree, because quasi-interpolation
by the Schoenberg B-spline operator is at most $\oo(\dt^2)$ \cite{deboor}.  This is in marked contrast to
discontinuous
polynomial collocation or Galerkin approximations of \eq{vie} which
converge at optimal order
\cite{brunnerbook,dgvie,dgvie2}.
However a simple modification of the B-spline basis near $t=0$ can give higher order stable approximations of \eq{vie}, as illustrated when $m=3$ in Sec.~\ref{sec:modbs}.

The \textbf{approximation error} $e_n(t)$ for $n>0$, $t \ge 0$ is
 \be
  e_n(t) = u(t_n-t) - \sum_{j=0}^n \vv_{n-j}\, b^m_{j-m}(t)\,,
\label{eq:approxerr}
\ee
where the coefficients $\vv_j$ satisfy
\be
  a(t_n) = \sum_{j=0}^n \q_j\, \vv_{n-j}
\label{eq:CSdefs}
 \ee
for weights $\q_j$ as defined in \eq{defbw}.
Note that $\vv_0 = 0$ (because $a(0) = 0$) and so the sums above can
be taken from $j=0$ to $n-1$, and it then follows from Property
\ref{bsprop:supp} that $e_n(t) = 0$ for $t \ge t_n$ and each weight can be written as $\q_j = \int_0^{t_n} K(t)\, b^m_{j-m}(t)\, dt$\,.  Hence, multiplying \eq{approxerr} by $K(t)$ and integrating gives
\[
  \int_0^{t_n} K(t)\, e_n(t)\, dt = \int_0^{t_n} K(t)\, u(t_n-t)\, dt
  - \sum_{j=0}^{n-1} \q_j\, \vv_{n-j} = 0\,,
\]
by \eq{vie} and \eq{CSdefs}, i.e.\ $e_n$ is orthogonal to $K$ on
$(0,t_n)$.  The formal convergence result is as follows.

\begin{theorem}
Suppose that $m \ge 1$ and the conditions \eq{smoothness} and
\eq{zerobc} hold for $d \ge 4$. Then
 \be
  |e_n(t)| \le C\, \dt^2
\label{eq:errbd}
 \ee
for $t  \in [t_m, T]$, for some $C$ independent of $n$ and $\dt$.  If $m = 1$ then \eq{errbd} holds for $t  \in [0, T]$.
\label{thm:errbd}
\end{theorem}

\textbf{Remarks:}
\bi
\item The $m = 1$ case has been fully analysed \cite{brunnerbook} and is just included for completeness.

\item The restriction to second order convergence for $m > 1$ is a
fundamental aspect of quasi-interpolation by classical B-splines and not an artefact
of the proof, and is illustrated in Figure \ref{fig:convfig} when \eq{zerobc} holds with $d=4$.

\item Equation \eq{errbd} trivially holds for $t \ge t_n$ (because
$e_n(t) = 0$), and so it is enough to prove the result for $t \in [t_m, t_n)$ when $m > 1$, where $n \le T/\dt$.
\ei

%-----------------------------------------------------------------------------
\begin{figure}
\begin{center}
\includegraphics[width=0.8\textwidth]{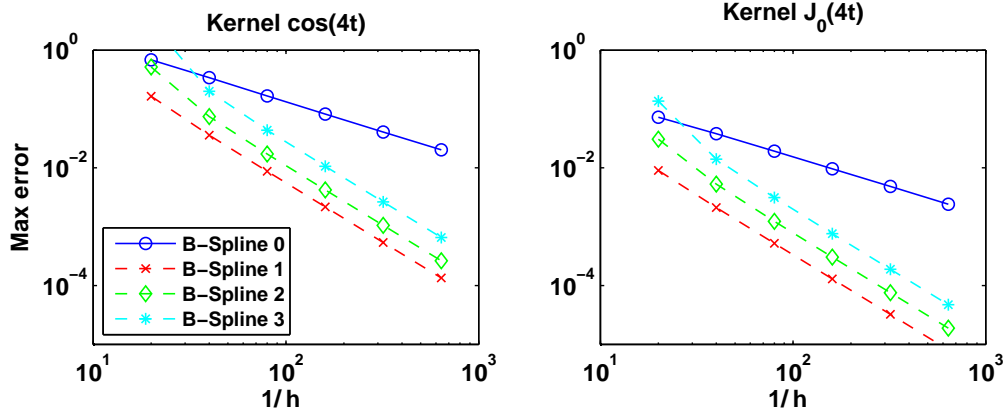}
\end{center}
\caption{\label{fig:convfig} Convergence results for the approximation of \eq{vie}
for two smooth kernel functions with maximum time $T=10$ and
$a(t) = t^6\exp(-50(t-1/2)^2)$.  Convergence rates of
$\oo(\dt^2)$ for splines of degree $m\ge 1$ and $\oo(\dt)$ for $m=0$ are clear.  Stability results for the highly oscillatory kernels $\cos \w t$ and $J_0(\w t)$ where the frequency $\w$ can be $\oo(1/\dt)$ are given in Section \ref{sec:BSstab.bessel} and Figure \ref{fig:stabJ0andcos}. }
\end{figure}
%-----------------------------------------------------------------------------

\begin{proof}
We first express $e_n(t)$ in terms of coefficients
$\eps_k := u(t_{k+(m-1)/2}) - \vv_k\,.$
Substituting the quasi-interpolation result
\eq{QIresult2} with $f(t) = u(t_n - t)$ in the definition \eq{approxerr} of
$e_n(t)$ gives
\[
  e_n(t) = \sum_{j=m}^n \eps_{n-j}\, b^m_{j-m}(t) +
  \sum_{j=n+1}^{n+m} u(t_n - t^m_{j-m})\, b_{j-m}^m(t) +
  \oo(\dt^2)\,,
\]
where we have used $b^m_{j-m}(t) = 0$ for $j < m$ and $j > n+m$.  It follows from
the assumptions \eq{smoothness} and \eq{zerobc} that $u(c\dt) = \oo(\dt^{d})$
for any constant $c$.  This implies that the second sum term in the
previous equation is $\oo(\dt^{d})$, and hence yields
\[
  e_n(t) = \sum_{j=m}^n \eps_{n-j}\, b^m_{j-m}(t) + \oo(\dt^2)
\]
for $t \in [t_m, t_n)$ with $t_n \le T$.  Because there are at most $m+1$ nonzero terms in this sum
for any $t$, it is sufficient to show that there exists a
constant $C$ independent of $\dt$ such that
 \be
  |\eps_j| \le
     C \dt^{2} \quad\mbox{for all $j \le T/\dt$.}
  \label{eq:epsbd}
 \ee

To prove \eq{epsbd} note that it follows from \eq{CSdefs} that $\ds{\sum_{j=0}^n \q_j\, v_{n-j} =
\int_0^{t_n} K(t)\, u(t_n - t)\, dt}$ and so
 \be
  \sum_{j=0}^n \frac{\q_j}{\dt}\; \eps_{n-j} =
  \sum_{j=0}^n \frac{\q_j}{\dt}\; u(t_{n-j+(m-1)/2})  -
  \frac{1}{\dt} \int_0^{t_n} K(t)\, u(t_n - t)\, dt := R_n\,.
\label{eq:defRn}
 \ee
If $\dt$ is sufficiently small, then expanding $K(t)$ and using
\ref{bsprop:mom}  gives
 \be
   \frac{\q_j}{\dt} = \left\{\begin{array}{ll}
    \ds{\frac{j+1}{m+1} + \dt\, K'(0)\,\frac{(j+1)^2(j+2)}{2(m+1)(m+2)} + \oo(\dt^2)}
      & \quad\mbox{for $j = 0:m-1$}\\
    K(t_{j-m}) + \frac{1}{2}\, \dt\,(m+1)\, K'(t_{j-m}) + \oo(\dt^2) & \quad\mbox{if $j \ge m$\,.}
  \end{array}\right.
\label{eq:wvalvar}
 \ee
It then follows from the quasi-interpolation result \eq{QIresult4}
with $p=3$ that when $n \ge m$,
\[
  \dt\, R_n = \sum_{\ell = 0}^{n-1} \sum_{k=2}^3 \eta_\ell^k\,
  u^{(k)}(t_{n-\ell}) + \oo(\dt^4)\ \ \mbox{for}\ \
\eta_\ell^k = \frac{(-1)^{k+1}}{k!}\, \int_{t_\ell}^{t_{\ell+1}} \!\!K(t)\, \left\{
  (t-t_\ell)^k -  \sum_{j=\ell}^{\ell+m} (t_{j-m}-t_\ell)^k\,  b^m_{j-m}(t) \right\}\,
  dt\,.
\]
It is then straightforward to show $R_{n+1} - 2\, R_n + R_{n-1} = \oo(\dt^3)$
and after some manipulation using \eq{wvalvar} the second central
difference of \eq{defRn} can be written as
 \be
  \left(1 + \dt\, \mu_o\right)\, \eps_{n+m+1} =
  \eps_n + \dt\, \sum_{\ell = 0}^m \mu_{\ell + 1}\,
  \eps_{n-\ell+m} + \dt^2 \sum_{\ell=m+1}^n \mu_{\ell+1}\,
  \eps_{n-\ell+m} + \gamma_n\,,
\label{eq:epsrec}
 \ee
where $\gamma_n = \oo(\dt^3)$ and all the $\mu_\ell$ are bounded.
This can be written as a one-step recurrence for the vector
$\ul{\delta}^n \in \R^{m+1}$ with components $\delta_j^n \equiv
\eps_{n+j}$ for $j = 0:m$.  The recurrence is $\delta_j^{n+1} =
\delta_{j+1}^n$ for $j = 0:m-1$ with $\delta_m^{n+1}$ given by
\eq{epsrec}, which gives the matrix--vector system
 \be
  \ul{\delta}^{n+1} = \left(M + \dt\, W_0\right)\, \ul{\delta}^n +
  \dt^2 \sum_{\ell=0}^{n-1-m} W_{n-\ell}\; \ul{\delta}^{m+\ell} +
  \gamma_n\, \ul{e}^m
\label{eq:vecdel}
 \ee
where $\ul{e}^m = (0,\dots,0,1)^T$, each matrix $W_\ell$ is bounded
and $M \in \R^{(m+1)\times(m+1)}$ is the circulant matrix whose only
nonzero components are $M_{m,0} = M_{j,j+1} = 1$ for $j=0:m-1$.
The eigenvalues of $M$ are the $(m+1)-$th roots of unity, and are
hence distinct.  Following Brunner \cite{Br1978} we note that $M$
belongs to Ortega's \cite[\S1.3]{ortega} Class M, and so there is a
vector norm $\|\cdot\|_*$ on $\R^{m+1}$ for which the induced matrix
norm satisfies $\|M\|_* = \rho(M) = 1$. Taking this norm of
\eq{vecdel} then implies that there is a constant $C$ such that
\[
  \|\ul{\delta}^{n+1}\|_* \le (1 + C\,\dt)\, \|\ul{\delta}^n\|_* + C\,
  \dt^2\, \sum_{\ell=0}^{n+1-m} \| \ul{\delta}^{m+\ell}\|_* +C\,\dt^{3}\,
\]
and the top bound of \eq{epsbd} gives $\|\ul{\delta}^0\|_* \le C\,
\dt^{d}$.  Standard arguments can then be used to show that
$\|\ul{\delta}^n\|_* \le \nu_n$ where $\nu_0 = C\, \dt^{d}$ and
\be\label{eq:nudiff}
  \nu_{n+1} = (1 + C\,\dt)\, \nu_n + C\,\dt^2\, \sum_{\ell=0}^{n}
  \nu_\ell + C\, \dt^3 \quad \mbox{for $n \ge 0$}.
\ee
This has the solution $\nu_n = A_+\, \lambda_+^n + A_-\,
\lambda_-^n$ where $\lambda_\pm = 1 + \oo(\dt)$ and $A_\pm =
\oo(\dt^2)$.  Hence $\nu_n \le C_1\, \dt^2\, e^{C_2\, T}$ for $n \le
T/\dt$, which concludes the proof of \eq{epsbd} and hence \eq{errbd}.
\end{proof}

%=====================================================================
%=====================================================================
\newsec{Modified cubic spline basis functions for \eq{vie}}
\label{sec:modbs}

The results of the previous section illustrate that the convolution spline framework can be used to derive new VIE approximations and prove their stability in cases not covered by standard convergence analysis (such as discontinuous or highly oscillatory kernels), but the restriction to second order convergence for the B-spline basis \eq{phibs} is not competitive with RK-based CQ methods \cite{Ba2010,BaLu2011,BaLuMe2011}.
We now show how a slight modification of the $m=3$ B-spline basis near $t=0$ can yield methods which are more accurate and have better stability properties than \eq{phibs}.

%=====================================================================
\subsection{Derivation}
\label{sec:modbs.deriv}

It is simpler to define the modified basis functions in terms of B-splines centred at zero, and we set $B(t) = b^3(t+2)$, so supp$(B) = (-2,2)$.  For $j \ge 3$ the basis functions are $\phi_j(t) = B(t-j)$, and we choose $\phi_j$ for $j=0:2$ to satisfy the parabolic runout conditions at $t=0$, namely
\[
  \phi_0(t) = B(t) + 3\, B(t+1)\,,\ \ \phi_1(t) = B(t-1) - 3\, B(t+1)\,,\ \ \phi_2(t) = B(t-2) +  B(t+1)\,.
\]
This means that
\be
  \sum_{j=-1}^\infty j^r \left(\phi_j(t) - B(t-j)\right) = 0 \quad\mbox{for $t \in [0,\infty)$ for $r=0:2$}\,,
\label{eq:bfsum}
\ee
where we set $\phi_{-1}(t) \equiv 0$, and in particular it ensures that quasi-interpolation in terms of $\left\{\phi_j\right\}_{j=0}^\infty$ is linearity-preserving on $[0,\infty)$.  Weights $q_j$ for the VIE \eq{vie} are given in terms of these basis functions by \eq{wtdef} and the approximation at $t= t_n$ is
\be
  \U_n(t_n-t) = \sum_{j=0}^{n} \vv_{n-j}\, \phi_j(t/\dt)\,,
\label{eq:vvcof}
\ee
where the $\vv_{j}$ coefficients are given by \eq{CSscheme}.  The individual coefficients can be directly obtained from $\U_n(t_n-t)$ by introducing dual basis functions $\db_k(t)$ such that
\[
  \int_0^\infty \phi_j(t)\, \db_k(t)\, dt = \delta_{j,k}\,.
\]
There are many ways in which a suitable dual basis can be chosen, and one possibility is to use continuous piecewise cubic functions on $[0,\infty)$ defined with respect to knots in $\N$.  For $k \ge 3$, we set $\db_k(t) = \db(t-k)$ where $\db$ is the even continuous piecewise cubic function on $[-2,2]$ which is zero at $\pm 2$ and is $C^1$ at the interior knots $\pm 1$, $0$ and satisfies
\[
  \int_{-2}^2 \db(t-\ell)\, B(t)\, dt = \delta_{0,\ell}\quad \mbox{for $\ell = 0:3$},
\]
and it is also possible to find suitable continuous piecewise cubics $\db_k$ with support in $[0,k+2)$ for $k=0:2$.  Calculating these dual basis functions in an algebraic manipulation package is straightforward, although their coefficients are messy and we omit their details.  Once the $\db_k$ have been obtained it can be shown by Taylor expanding that for any sufficiently smooth function $f$,
\be
  \int_{\max(0,k-2)}^{k+2} f(t_n - \dt s)\, \db_k(s)\, ds = f(t_{n-k}) -\frac{\dt^2}{6}\, f''(t_{n-k}) + \oo(\dt^4)\,.
\label{eq:dbexp}
\ee

Multiplying \eq{vvcof} by $\db_k(t/\dt)$ and integrating over $[0, \infty)$ gives
\[
  v_{n-k} = \int_{\max(0,k-2)}^{k+2} \U_n(t_n - \dt s)\, \db_k(s)\, ds
\]
and we now use this representation and \eq{dbexp} in order to obtain the best approximation of the solution $u$ of \eq{vie} in terms of the basis functions $\phi_j$ when \eq{smoothness} holds with $d = 4$.  Specifically, we set
\be
  u(t_n - t) = \sum_{j=0}^{n} u_{n-j}\, \phi_j(t/\dt) + \tilde{R}_n(t)\,,
\label{eq:usolrep}
\ee
where
\be
  u_{n-j} := \int_{\max(0,j-2)}^{j+2} u(t_n - \dt s)\, \db_j(s)\, ds = u(t_{n-j}) -\frac{\dt^2}{6}\, u''(t_{n-j}) + \oo(\dt^4)\,.
\label{eq:defuj}
\ee
We show that this scheme is fourth order accurate, and discuss its stability properties for non-smooth kernels.

%=====================================================================
\subsection{Convergence}
\label{sec:modbs.conv}

It follows from \eq{usolrep}--\eq{defuj} that the approximation error $e_n(t) := u(t_n - t) - \U_n(t_n -t)$ for this scheme is
\[
  e_n(t) = \sum_{j=0}^n \eps_{n-j}\, \phi_j(t/h) + \tilde{R}_n(t)
\]
where $\eps_k = u_k - v_k$.  We now show that modifying the basis functions near $t = 0$ as described above improves the scheme's accuracy from second to fourth order.

\begin{theorem}
Suppose that the conditions \eq{smoothness} and
\eq{zerobc} hold for $d \ge 6$.  Then there exists a constant $C$ independent of $n$ and $\dt$ such that if $t_n \le T$, the VIE approximation error satisfies
\be
  |e_n(t)| \le C\, \dt^4
\label{eq:moderrbd}
 \ee
for $t  \in [t_1, t_n]$.
\label{thm:moderrbd}
\end{theorem}

\begin{proof}
It is straightforward to verify that if $s \in [0,1)$,
\[
  \tilde{R}_n(t_k + \dt s) = u(t_{n-k} - \dt s) - \sum_{\ell = -1}^2 u_{n-k-\ell}\; \phi_{k+l}(k+s)
\]
where $\phi_{-1}$ is taken to be zero, and hence it follows from \eq{bfsum}, \eq{defuj} and standard B-spline properties that
\[
  \tilde{R}_n(t_k + \dt s) = \left\{\begin{array}{l}
    \dt^3(1-s)^3\, u'''(t_n)/6 + \oo(\dt^4)\quad\mbox{when $k=0$}\\
    \oo(\dt^4)\quad\mbox{when $k\ge 1$\,.}
  \end{array}\right.
\]
It is thus sufficient to show that each $\eps_k = \oo(\dt^4)$ (because at most four basis functions are nonzero for any $t$).  The proof follows that of Thm.~\ref{thm:errbd}, and the expression analogous to
\eq{defRn} is
\[
  \sum_{j=0}^{n-1} \frac{\q_j}{\dt}\; \eps_{n-j} =
  -\sum_{k=0}^{n} \int_0^{1} K(t_k + \dt s)\, \tilde{R}_n(t_k + \dt s)\, ds + \oo(\dt^5)\,. % := \rho_n\,.
\]
Taking the second central difference, using the fact that when $K \equiv 1$ the scheme coefficients are
\be\label{eq:qmod}
  q_0 = 5\dt/8\,,\quad q_1 = 5\dt/6\,,\quad q_2 = 25\dt/24\,, \quad\mbox{and\quad $q_j = \dt$ for $j \ge 3$}
\ee
gives an expression like \eq{vecdel} with $m=3$ where the bottom row of the matrix $M$ is now $[-1, 6, 0, 10]/15$.  The eigenvalues of $M$ are all distinct and its spectral radius is $\rho(M) = 1$, and this again yields a bound $\nu_n$ which satisfies a difference scheme like \eq{nudiff} whose final term is $C\, \dt^5$.
This gives an $\oo(\dt^4)$ bound for each
$|\eps_k|$ with $t_k \le T$, and \eq{moderrbd} follows.
\end{proof}

%=====================================================================
\subsection{Stability for discontinuous and highly oscillatory kernels}
\label{sec:modbs.stab}

We first consider the discontinuous kernel introduced in Section \ref{sec:BSstab.disc}.  The duration
$L = (\Md+r)\dt$ is fixed independent of $\dt$ such that integer
$\Md\ge 5$ and $r \in [0,1)$. In this case the $q_j$ are given by \eq{qmod} for $j=0:\Md-2$, the coefficients $q_j$ for $j=\Md-1:\Md+2$ are polynomial in $r$ and satisfy
\[
    q_{\Md-2}=\dt\ \ge q_{\Md-1} \ge q_{\Md} \ge q_{\Md+1} \ge q_{\Md+2} \ge 0\,,
\]
and $q_j = 0$ for $j \ge \Md +3$.

The stability proof follows that of \cite[\S 3.1.3]{cq2}; forward differencing \eq{defpn} gives
\[
       15\,p_n + 5\,p_{n-1} + 5\,p_{n-2} - p_{n-3} + 24 \sum_{j=\Md-1}^{\Md+3}\frac{(q_j-q_{j-1})}{\dt}\;p_{n-j} = 0
       \quad \mathrm{for}\ n \ge 2,
\]
where the first stability factors are $p_0 = 1$, $p_1 = -q_1/q_0 = -4/3$ and we set
$p_j=0$ for $j<0$.  Similar manipulations to those in the previous subsection then give
\[
  |p_n| \le \frac{11}{15} \,\cumulmax_{n-1} + \frac{8}{5} \,\cumulmax_{n-(\Md-1)}
      \quad \mathrm{for}\ n \ge 2\,
\]
where now $\cumulmax_k = \max_{j\le k}|p_j|$\,.  It then follows from similar arguments to those used in \cite[\S 3.1.3]{cq2} that if $n \ge \Md-1$ then
\[
  \cumulmax_n \le \frac{11}{15} \,\cumulmax_{n-1} + \frac{8}{5} \,\cumulmax_{n-(\Md-1)} \le \frac{11}{15} \,\cumulmax_{n} + \frac{8}{5} \,\cumulmax_{n-(\Md-1)}\,,
\]
i.e.\ $\cumulmax_n \le 6 \, \cumulmax_{n-(\Md-1)}$, and $|p_n| \le \cumulmax_n = 4/3$ for $n = 1:\Md-2$.  In combination these give the stability bound
\[
   |p_n| \le \cumulmax_n \le C\, 6^{n/(\Md-1)} \le C\, 6^{(T+1)/L}
\]
for all $n \le T/\dt$ when $\dt$ is sufficiently small, and so
the stability coefficients are bounded
independently of $\dt$ as required by Def.~\ref{def:stability}.  Note that this bound is very pessimistic, and in practice the factor is about $1.5^{T/L}$.

This scheme is also stable when applied to the oscillatory kernels $J_0(\w t)$ and $\cos(\w t)$ examined in Section \ref{sec:BSstab.bessel}.
The analysis is much simpler to carry out for this scheme, since there is no problem with poles of $1/Q(\xi)$ on or near the unit circle $|\xi|=1$.
We find that
\[
     |p_n| \le C e^{\sigma T}
      \quad \mathrm{where} \quad
      C = \left\{ \begin{array}{rc} 2.2\,, & K(t)=\cos(\w t) \\
                         1.5\,, & K(t)=J_0(\w t)
                       \end{array}\right.
\]
for all $\w \dt \in [0,20\pi]$ and beyond.
This is verified in tests computing $p_n$ directly from \eq{defpn} for a finite number of steps $n \le 2500$ and the same range of values of $\w\dt$. They indicate that $|p_n| \le 1.82$ (for $\cos(\w t)$) and $|p_n| \le 4/3$ (for $J_0(\w t)$), consistent with the
estimate above.

%=====================================================================
\subsection{Numerical results}
\label{sec:modbs.numres}

%--------------------------------------------------------------------------------
\begin{figure}[h!]
\begin{center}
\includegraphics[width=0.75\textwidth]{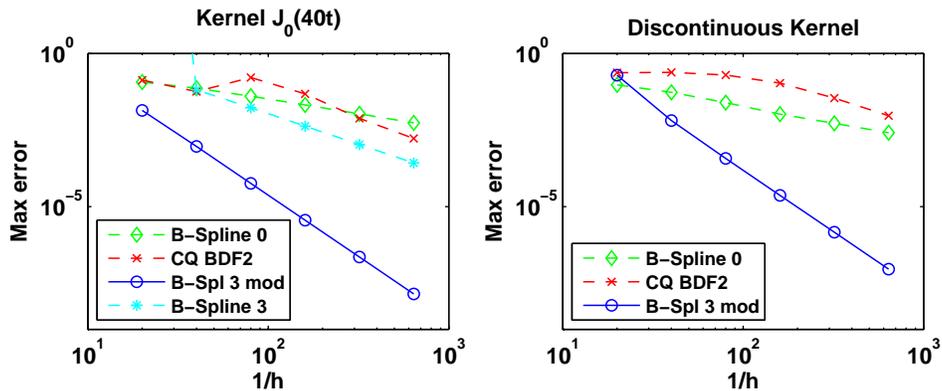}
\end{center}
\caption{\label{fig:viecompare} Convergence results for the approximation of \eq{vie}
with maximum time $T=10$ and $a(t) = t^6 \exp(-50(t-1/2)^2)$ (see text for details).  The stability and $\oo(\dt^4)$ convergence rate of modified B-spline 3 is clear for both problems.}
\end{figure}
%--------------------------------------------------------------------------------

Numerical comparisons of the new modified cubic B-spline approximation for \eq{vie} are compared with the CQ BDF2 method and with the convolution B-splines with $m=0$, $3$ from Sec.~\ref{sec:bspline} in
Figure~\ref{fig:viecompare} when $a$ satisfes \eq{zerobc} with $d=4$.
The convergence rates for a
smooth kernel (left subplot) are as expected:
CQ BDF2 has rate $\oo(\dt^2)$; B-splines 0 and 3 have rates $\oo(\dt)$ and $\oo(\dt^2)$;
and the modified B-spline 3 has rate $\oo(\dt^4)$.
The discontinuous kernel problem (right subplot) has the same convergence rates, apart from the isogeometric $m=3$ approximation which was shown to be unstable in Section
\ref{sec:BSstab.disc}, and is not illustrated.

%=====================================================================
%=====================================================================
\newsec{Convolution splines for TDBIEs}
\label{sec:tdbie}

The results shown here are obtained by approximating the solution of the TDBIE \eq{tdbie} in space by piecewise constant basis functions on a generally irregular triangular grid, and in time by convolution splines as
\[
   u(\vecx,t_n-t') \approx \sum_{j=0}^n \sum_{k=1}^{M} u_k^{n-j}\, \phi_j(t'/\dt)\, \eta_k(\vecx)\,,
\]
where the $\phi_j$ are the modified cubic B-splines from Sec.~\ref{sec:modbs.deriv},
\[
  \eta_k(\vecx) = \left\{ \begin{array}{ll} 1 &  \mbox{for $\vecx \in \ele_k$}\\ 0 & \mathrm{otherwise} \end{array} \right.
\]
and $\ele_k$ is the $k$th triangle on the surface $\Gamma$.
The spatial Galerkin formulation of the problem gives the time-marching scheme \eq{BIEscheme} with
\[
    \tens{Q}_{j,k}^m = \int\!\!\!\int_{\ele_j \times \ele_k} \!\!\! \frac{\phi_m(|\vecx-\vecy|/\dt)}{|\vecx-\vecy|}\; d\vecx\, d\vecy,
    \quad \mbox{and} \quad
    a_j^n = \int_{\ele_j} a(\vecx,t^n)\, d\vecx \,.
\]
The matrices $\tens{Q}^m$ are symmetric and calculating each component involves a four dimensional integral.
The off-diagonal components and the components $a_j^n$ have smooth integrands and are approximated by a composite triangular quadrature with
16 sub-triangles, each of which is fourth order with 6 quadrature points.
Each diagonal element of the $\tens{Q}^m$ has a singular integrand, and is first converted into smooth subintegrals using
a Duffy-type transformation, and then approximated by the same quadrature rule
as the rest of the calculation.

The piecewise constant spatial approximation is globally first order accurate (i.e.\ $\oo(\dx)$), but there is local (second order) superconvergence at the element midpoints, and this is exploited in the figures below.  In particular, $\|\cdot\|_\infty$ is the discrete $L_\infty$ norm measured at element midpoints.  Similar results are obtained when piecewise linear spatial basis functions are used (this is globally second order).  A more accurate spatial approximation could potentially be used (to take advantage of the fourth order accuracy in time), along with higher order quadrature.

%--------------------------------------------------------------------------------
\begin{figure}[ht]
\begin{center}
\includegraphics[width=0.75\textwidth]{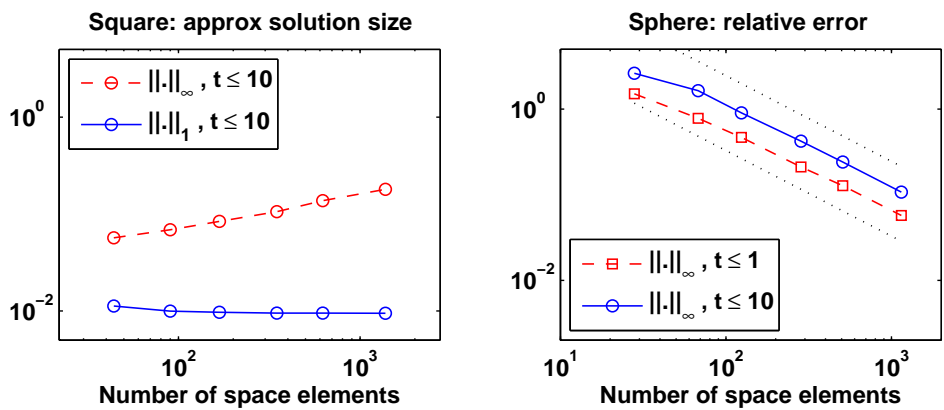}
\end{center}
\caption{\label{fig:squareandsphere} Solution plots for the TBIE \eq{tdbie} when the scattering surface $\Gamma$ is a unit square plate (left plot) and a unit sphere (right plot).  See text for details of the space and time approximations used.}
\end{figure}
%--------------------------------------------------------------------------------
%--------------------------------------------------------------------------------
\begin{figure}[ht]
\begin{center}
\includegraphics[width=0.75\textwidth]{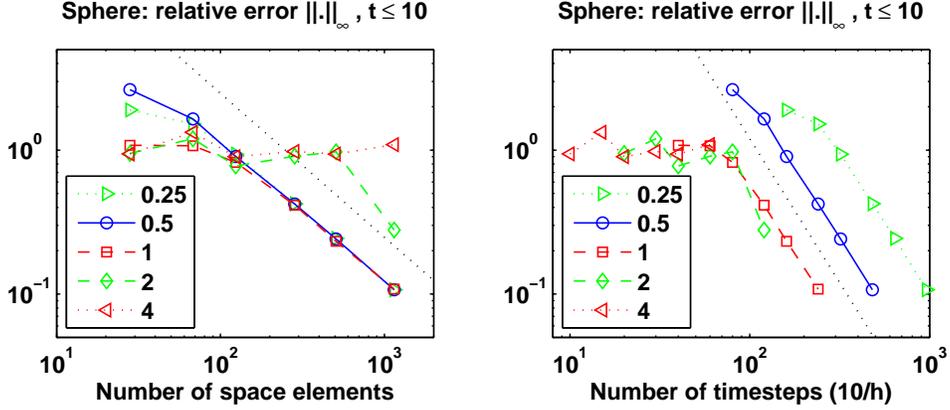}
\end{center}
\caption{\label{fig:ratvar} Solution plots for the TBIE \eq{tdbie} for mesh ratios $\dt/\dx$ ranging from 0.25 to 4 when the scattering surface $\Gamma$ is a unit sphere. The dotted line indicates second order convergence.  See text for details of the space and time approximations used.}
\end{figure}
%--------------------------------------------------------------------------------

Figure~\ref{fig:squareandsphere} shows results for \eq{tdbie} when $\Gamma$ is a flat plate and a sphere, both with incident field $a(\vecx,t) = a_0(t+t_0-|\vecx|)/|\vecx|$, where $a_0(t) = t^4 \exp(-20(t-1/2)^2)$ for $t > 0$.  In these two tests the mesh ratio is chosen to be $\dt/\dx = 1/2$\,, where $\dx$ is the size of a typical space mesh element.
The left-hand graph shows the size (and hence stability) of the approximate solution on a
unit square plate.
The growth in the $L_\infty$ norm is due to a corner singularity in the
exact solution \cite{HoMaSt} while the maximum of the 1-norm is well-behaved as the mesh is refined.
The right-hand plot shows the maximum relative error (i.e.\ the error normalised by the maximum size of the solution) for scattering
from a unit sphere -- an exact solution for this problem is given in \cite{SaVe2011a}.
The dotted lines show a second order convergence slope, and this is the best which can be expected
from the spatial approximation, despite the higher order accuracy of the temporal approximation.

Figure \ref{fig:ratvar} demonstrates the impact of changing the mesh ratio $\dt/\dx$ in the sphere scattering example described above. The left plot shows the maximum relative error in the solution
against the number of space elements, and the right plot shows the dependence of this error on $\dt$. As one would expect for a scheme with higher order accuracy in time than space, the time error decreases faster than the space error when the mesh is refined with fixed mesh ratio, and the space error eventually dominates.
This is clear on the left plot for ratios $\dt/\dx = 0.25,0.5$ and $1$.
For the larger mesh ratios the time step size is simply too big to resolve the input function accurately, and the asymptotic convergence regime has not yet been reached.
Increasing the mesh ratio decreases the number of time steps (and hence matrices $\tens{Q}^m$) used, but increases the number of non-zero entries in each matrix.
However the net result is a decrease in the computational cost for both the time marching calculation and the computation of the matrices $\tens{Q}^m$.

%--------------------------------------------------------------------------------
\begin{figure}[ht]
\begin{center}
\includegraphics[width=0.75\textwidth]{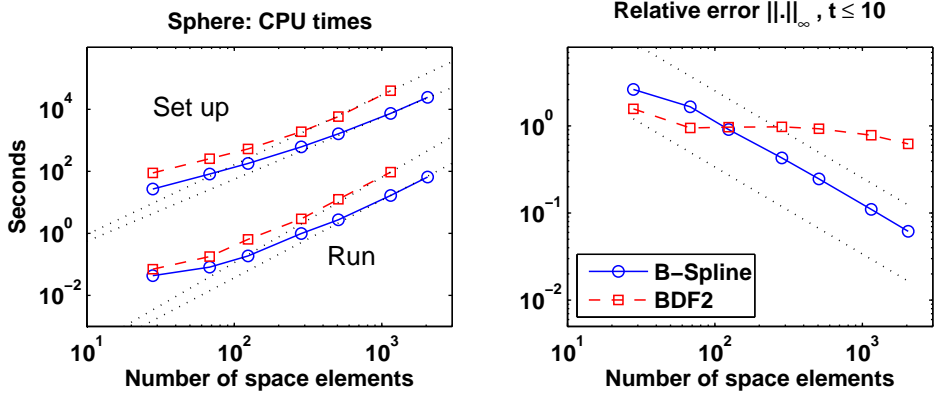}
\end{center}
\caption{\label{fig:cpu-err} Solution plots for the TBIE \eq{tdbie} with time approximation given by convolution cubic splines (blue solid) and CQ based on BDF2 (red dashed) when $\Gamma$ is a unit sphere. \textbf{Left plot:} shows the CPU time in seconds for computing the system matrices (set up) and the algorithm run time -- the dotted lines represent the asymptotic computational complexity $\oo(N^p)$ for each approach described in the text.  \textbf{Right plot:} shows the relative error vs $N_S$ for each method, and the dotted lines indicate second order convergence.}
\end{figure}
%--------------------------------------------------------------------------------

The final set of results compares the performance of convolution cubic splines for time-stepping TDBIEs with that of CQ based on BDF2.  The \textbf{set-up time} for each method is proportional to the number of non-zero entries in the $\tens{Q}^m$ matrices of \eq{BIEscheme} (with sufficiently small entries in the CQ case ignored).  For convolution spline (and space-time Galerkin) methods this is $\oo(N^4)$ when $\dt \approx \dx = 1/N$.  The support of the CQ basis functions grows with $m$, as illustrated in Fig.~\ref{fig:basis_BBDF2} (see also \cite{HaKrSa2009}), and the set-up cost for this method is $\oo(N^{4.5})$.  As described in Sec.~\ref{sec:intro.tdbie}, the \textbf{run time} for the basic convolution spline schemes (i.e.\ without using a plane-wave or other fast method to speed it up) is $\oo(N^5)$, and for the basic CQ approach it is $\oo(N^{5.5})$.  The left-hand plot of Fig.~\ref{fig:cpu-err} shows a graph of CPU time against the number of space elements and the dotted lines in each case show the asymptotic computational complexity $\oo(N^p)$ as tabulated below.  Note that although for each method the run time is clearly growing faster with $N_S \approx N^2$ than the set up time, the set-up time dominates for problems of moderate size.

%--------------------------------------------------------------------------------
\begin{table}[h]\begin{center}

\begin{tabular}{l|c|c}
Method & Setup & Run\\ \hline
{spline} & $\oo(N^4)$ & \st$\oo(N^{5})$\\
{CQ} & \st$\oo(N^{4.5})$ & $\oo(N^{5.5})$\\
\end{tabular}

\caption{Computational complexity $\oo(N^p)$ of the set up and run times for algorithm \eq{BIEscheme} with with time approximation given by convolution spline and CQ, where $\dt \approx \dx = 1/N$.}

\end{center}\end{table}

%--------------------------------------------------------------------------------

The right-hand plot of Fig.~\ref{fig:cpu-err} shows the approximation error for the two schemes, with dotted lines of slope $\oo(\dx^2)$.  The far superior accuracy of the convolution spline approximation is clear, and the poor performance of the BDF2 CQ method is because
it is highly damped -- although it is second order convergent, the mesh has to be very fine for this to be apparent over a long time calculation.  This is further illustrated in Fig.~\ref{fig:sol}, which shows the potential $u$ calculated on the surface of the sphere vs time when $N_S = 508$.  The convolution spline approximation matches the exact solution (which is independent of $\vecx$, see \cite{SaVe2011a}) extremely well.

%--------------------------------------------------------------------------------
\begin{figure}[h]
\begin{center}
\includegraphics[width=0.8\textwidth]{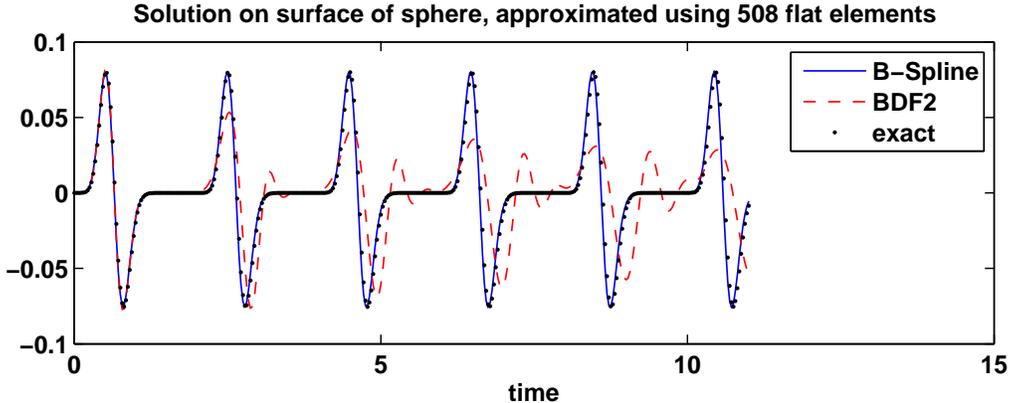}
\end{center}
\caption{\label{fig:sol} Approximate solution $u$ of \eq{tdbie} plotted against time when $\Gamma$ is a unit sphere.  The exact solution is independent of $\vecx$ and is shown as black dots, and the approximate solution is calculated using convolution spline (blue solid) and CQ based on BDF2 (red dashed).}
\end{figure}
%--------------------------------------------------------------------------------

%=====================================================================
%=====================================================================
\newsec{Conclusions}
\label{sec:conc}

We have derived a new framework for time-stepping approximations of the TDBIE \eq{tdbie}.  The system matrices $\tens{Q}^m$ of the resulting scheme \eq{BIEscheme} have the same degree of sparsity as for a space--time Galerkin approximation, but are much more straightforward to calculate, especially when higher order (smoother) B-splines are used.  The method is constructed as an approximation scheme for the VIE \eq{vie}, and key properties
are its backward time aspect \eq{CQapprox}, that the basis functions have compact support and are (mainly) translates, and they satisfy the sum to unity condition \eq{pou}.

These properties permit a full stability analysis for VIEs with kernels which capture some of the important properties of TDBIE problems, as illustrated for B-spline basis functions in Secs.~\ref{sec:bspline.stab} and \ref{sec:modbs.stab}.  In particular the analysis gives a stability bound for the $p_n$ coefficients for the Bessel function kernel \eq{BFvie} when $\w \approx 1/\dt$.  This means that the convergence proof in \cite{DaDu2003} for the TDBIE \eq{tdbie} on $\Gamma = \R^2$ could be applied to an approximation which is a Fourier interpolant in space and a convolution spline in time, without the need to impose an additional (essentially uncheckable) stability assumption.
The modified cubic convolution spline approximation of Sec.~\ref{sec:modbs} is fourth order accurate and gives a very stable approximation for discontinuous and highly oscillatory VIE kernels.  The TDBIE numerical test results indicate that the scheme is very stable and performs well.

There is current interest in TDBIE time-stepping methods which can use variable time-steps.
See e.g.\ \cite{Lo-FeSa2011,Lo-FeSa2012} for convolution quadrature methods and \cite{glafke} for
space and time adaptation in the full space-time Galerkin method for scattering problems in 2D space.
Because our TDBIE time-stepping method is based on B-splines (whose key approximation properties are retained for a non-uniform knot distribution) and standard piecewise polynomials in space, the various strategies described in \cite{glafke} for space and time adaptation could be applied.
However, using variable time-stepping in any TDBIE approximation algorithm imposes an overhead because it essentially involves recalculating the $\tens{Q}^m$ matrices of \eq{BIEscheme} at every time step, and this is extremely expensive, as illustrated in the left-hand plot of Fig.~\ref{fig:cpu-err}.

We note that there are also other choices of basis functions which seem to give stable approximations of \eq{vie} and \eq{tdbie} when used in the same convolution framework.  Non-polynomial temporal basis functions $\phi_j$ are introduced in \cite{cq2}; they are translates for $j \ge 2$, and for $j = 0:1$ they are defined as described in \cite{schaback} for radial basis function (RBF) multi-quadrics in order to ensure that quasi-interpolation in terms of $\left\{\phi_j\right\}_{j=0}^\infty$ is linearity-preserving.  They also work well as a temporal approximation of the TDBIE \eq{tdbie}, but because the basis functions are global the system matrices need to be sparsified (but this is straightforward because they are highly peaked).
The method derived in \cite{cq2} is second order accurate, and the fourth order modified
cubic B-spline approximation of Sec.~\ref{sec:modbs} is a significant improvement.  Extension of this approach to modified B-splines with $m > 3$ is work in progress.

%=====================================================================
%=====================================================================

\bibliographystyle{plain}
\bibliography{cqbib}

%=====================================================================
\end{document}